\documentclass[12pt,reqno]{amsproc}

\usepackage[margin=.9in]{geometry}
\usepackage[clock]{ifsym}
\usepackage{phaistos}
\usepackage[T1]{fontenc}
\usepackage{bbm}
\usepackage[usenames]{color}
\usepackage{mathpazo}
\usepackage{amsmath,amssymb,amsthm} 
\usepackage{wrapfig}
\usepackage{tikz-cd}
\usepackage[shortlabels]{enumitem}
\usepackage{blindtext}
\usepackage{caption}
\usepackage{mathrsfs,bm,nicefrac}
\usepackage[all,cmtip]{xy}
\usepackage{listings}
\definecolor{PineGreen}{rgb}{0.0,0.47,0.44}
\definecolor{MidnightBlue}{rgb}{0.1,0.1,0.44}
\definecolor{magenta}{rgb}{1.0,0.0,1.0}
\definecolor{bl1}{HTML}{4479A1}
\definecolor{pur1}{HTML}{52196D}
\definecolor{mag1}{HTML}{2AD0F1}
\definecolor{org1}{rgb}{.92,.39.21}
\definecolor{pur2}{rgb}{.53,.47,.7}
\definecolor{myblue}{RGB}{72,127,227}
\definecolor{mygreen}{RGB}{48,142,48}
\usepackage{hyperref}
\hypersetup{
	colorlinks=true,
	linkcolor=bl1,
	citecolor=PineGreen,
	filecolor=magenta,
	urlcolor=bl1
}
\usepackage{cleveref}
\usepackage{twemojis}
\newcommand{\spc}{\hspace*{0.25in}}

\makeatletter
\newcommand{\eqnum}{\refstepcounter{equation}\textup{\tagform@{\theequation}}}
\makeatother

\newtheorem{theorem}{Theorem}
\numberwithin{theorem}{section}
\newtheorem{proposition}[theorem]{Proposition}
\newtheorem*{theorem*}{Theorem}
\newtheorem{lemma}[theorem]{Lemma}
\newtheorem{corollary}[theorem]{Corollary}
\theoremstyle{definition}

\newtheorem{definition}[theorem]{Definition}

\theoremstyle{remark}
\newtheorem{remark}[theorem]{Remark}
\newtheorem{example}[theorem]{Example}

\newcommand{\Con}{\mathbf{Con}}

\newcommand{\polar}{\text{\rm P}}

\newcommand{\RR}{\mathbb{R}}

\newcommand{\QQ}{\mathbb{Q}}
\newcommand{\PP}{\mathbb{P}}

\newcommand{\pp}{\mathbb{P}}
\newcommand{\CC}{\mathbb{C}}

\newcommand{\bF}{\mathbf{F}}
\newcommand{\Jel}{\mathbf{Jel}}

\newcommand{\bB}{\mathbf{B}}
\newcommand{\bV}{\mathbf{V}}
\newcommand{\ideal}[1]{\left\langle #1 \right\rangle}

\newcommand{\cP}{\mathscr{P}}

\newcommand{\bS}{\mathbf{S}}

\newcommand{\set}[1]{{\left\{{#1}\right\}}}
\newcommand{\ip}[1]{\left\langle{#1}\right\rangle}

\newcommand{\B}{\mathbf{R}}

\newcommand{\jac}{\text{\bf J}}
\usepackage[normalem]{ulem}
\makeatletter
\DeclareRobustCommand
\myvdots{\vbox{\baselineskip4\p@ \lineskiplimit\z@
		\hbox{.}\hbox{.}\hbox{.}}}
\makeatother

\newcommand{\inj}{\hookrightarrow}
\newcommand{\surj}{\twoheadrightarrow}
\usepackage[foot]{amsaddr}

\def\DD{D\kern-.7em\raise0.3ex\hbox{\char '55}\kern.33em}

\usepackage{booktabs}
\usepackage{colortbl}
\definecolor{Ftitle}{RGB}{11,46,108}
\colorlet{tableheadcolor}{Ftitle!25} 
\colorlet{tablerowcolor}{gray!10} 
\newcommand{\Pure}{\text{\rm Pure}}

\begin{document}
	
	\title{Effective Whitney Stratification of Real Algebraic  Varieties}
	\author{Martin Helmer} 
	\address[MH]{
		Department of Mathematics, Swansea University,
		Swansea, Wales, UK}\email{martin.helmer@swansea.ac.uk}
  \author{Anton Leykin}\address[AL]{School of Mathematics, Georgia Institute of Technology, Atlanta, USA}\email{leykin@math.gatech.edu}
	\author{Vidit Nanda}\address[VN]{Mathematical Institute,
		University of Oxford, Oxford, UK}\email{nanda@maths.ox.ac.uk}
\begin{abstract}
We describe new algorithms to compute Whitney stratifications of real algebraic varieties. Using either conormal or polar techniques, these algorithms stratify a complexification of a given real variety. We then show that the resulting stratification can be described by real polynomials. We also extend these methods to stratification problems involving the so-called full semialgebraic sets as well as real algebraic maps. 
\end{abstract}
    	
     \maketitle

\section{Introduction}
Real varieties arise more naturally in applications than their complex counterparts. Unfortunately, the geometry of a real variety is not readily revealed by its defining polynomials: for instance, every real variety is the vanishing locus of a single polynomial, whereas the minimal number of defining polynomials forms an important (and non-trivial) invariant for complex varieties. As such, one often seeks methods to subdivide a given variety into simpler regions which fit together coherently. And indeed, cylindrical algebraic decomposition \cite{arnon1984cylindrical,brown2007complexity,caviness2012quantifier,england2016experience,wilson2012speeding} provides precisely such a subdivision into contractible cells. One of our goals in this work is to outline what we view as a key step in an alternative framework  for studying broad classes of applied and computational problems in real algebraic geometry.

\subsection*{Whitney Stratifications} There are many advantages to obtaining a cell decomposition of a variety --- it becomes possible, for instance, to directly compute the Euler characteristic, and (with a bit more work) the Betti numbers. On the other hand, cells, being homeomorphic images of open balls, are rather rigid objects. Cylindrical algebraic decomposition tends to be impractical on larger (or more complicated) varieties precisely because one often has to perform many subdivisions in order to obtain cells; similarly, the number of these cells can also be doubly exponential relative to the number of variables \cite{brown2007complexity}. A more general and flexible paradigm is furnished by the stratifications introduced by Whitney \cite{whitney1965tangents}, which play an essential role in microlocal geometry and stratified Morse theory \cite{SMTbook,kashiwara-schapira1}. 

\begin{definition} A pair $(M,N)$ of smooth submanifolds of $\RR^n$ satisfies Whitney's {\bf Condition (B)} if the following property holds at every point $q \in N$. Given any pair of sequences $\set{p_k} \subset M$ and $\set{q_k} \subset N$ with $\lim p_k = q = \lim q_k$, if the limiting tangent space and the limiting secant line 
\[
T := \lim_{k \to \infty} T_{p_k}M \quad \text{and} \quad \ell := \lim_{k \to \infty}[p_k,q_k]
\]
both exist, then $\ell \subset T$. 
\end{definition}

A {Whitney stratification} of a subset $X \subset \mathbb{R}^n$ is any locally-finite decomposition 
\[
X = \coprod_{\alpha}M_\alpha
\] into smooth, connected nonempty manifolds $M_\alpha \subset X$ called {\em strata}, so that every pair $(M_\alpha,M_\beta)$ satisfies Condition (B). It is known that every (real or complex) algebraic variety $X$ admits a Whitney stratification such that for each dimension $i$ the union $X_i$ of all strata of dimension $\leq i$ forms a subvariety (see e.g.~\cite[page 36--38]{wall2006regular}). One may therefore identify any such stratification of a $k$-dimensional variety $X$ with a flag 
\begin{equation}\label{eq:WhitneyFlag}
    X_\bullet = \big(\varnothing \subset X_0 \subset X_1 \subset \cdots\subset  X_{k-1} \subset X_k = X\big)
\end{equation}
of subvarieties, with the implicit understanding that its $i$-strata are connected components of $X_i - X_{i-1}$.

\begin{figure}[h!]
    \includegraphics[scale=0.4]{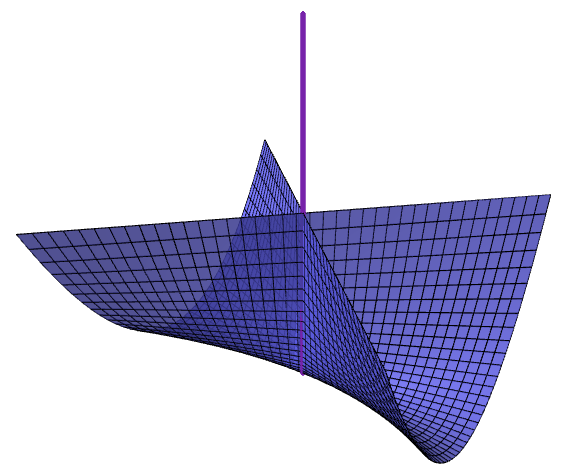} \caption{\small A depiction of the Whitney Umbrella, which is commonly used to illustrate equisingularity.
    \label{fig:stratSurf}}
\end{figure}

One key consequence of imposing Condition (B) on strata pairs is {\em equisingularity}: a sufficiently small tubular neighborhood around a given stratum forms a (locally trivial) stratified fiber bundle over that stratum \cite[Cor 10.6]{Mather2012}. As a consequence, small neighborhoods in $X$ of two points lying in the same stratum $M_\alpha$ are stratified homeomorphic\footnote{A homeomorphism $\phi:X \to Y$ of stratified spaces is said to be stratified whenever it maps each stratum $S \subset X$ diffeomorphically onto a stratum $\phi(S) \subset Y$.} to each other. Consider, for instance, the (real) {\em Whitney Umbrella} from Figure \ref{fig:stratSurf}. This is the variety $X \subset \RR^3$ defined by $x^2-y^2z=0$, whose singular locus is the (positive) $Z$-axis (depicted as a purple line in the figure). If we take a small neighbourhood in $X$ of a point $(0,0,z)$ lying on the $Z$-axis, then its topology changes depending on whether $z$ is positive, negative or zero. Thus, the origin must constitute a stratum separate from the rest of the $Z$-axis in any valid Whitney stratification of $X$.  

\subsection*{This Paper} The main contribution of our work here is a suite of practical algorithms for constructing Whitney stratifications of real algebraic varieties and related objects. Our prior work \cite{hn}\footnote{The Arxiv version \cite{hn} cited here incorporates both the original paper \cite{hnFOCM} and its subsequent correction \cite{hnFOCMCorrection} in one document.} describes an effective algorithm {\bf WhitStrat} for constructing Whitney stratifications of complex varieties; this algorithm relies on primary decompositions of conormal fibers (and hence, on Gr\"obner basis computation). Here we introduce a new algorithm {\bf WhitneyPolar} for computing complex stratifications; it is based on the theory of polar varieties \cite{FTpolar} and exhibits better complexity and runtimes than {\bf WhitStrat}.

In the introductory remarks to \cite{hn}, we highlighted the lack of Gr\"obner basis techniques over $\RR$ as a primary obstacle to performing similar stratifications for real algebraic varieties and semialgebraic sets. We overcome this obstacle here by examining, for any real variety $X \subset \mathbb{R}^n$, an associated complex variety $X(\CC)$ --- this is obtained by treating (any set of) the real polynomials which define $X$ as elements of the ring $\CC[x_1,\ldots,x_n]$. Having constructed $X(\CC)$, we obtain an effective Whitney stratification of $X$ via the following two steps:   
\begin{enumerate}
    \item we show in Theorem \ref{thm:RealWhitStrat} that if $X_\bullet(\CC)$ is any Whitney stratification of $X(\CC)$ for which each $X_i(\CC)$ is generated by real polynomials, then the corresponding real varieties $X_i$ form a valid Whitney stratification of $X$; and moreover,
    \item we show in Corollary \ref{cor:whitstratworks} that if a complex variety $Y \subset \CC^n$ is generated by real polynomials, then running either {\bf WhitStrat} or {\bf WhitneyPolar} on $Y$ produces a Whitney stratification $Y_\bullet$ for which each $Y_i$ is also generated by real polynomials.
\end{enumerate}
These two results open the door for stratifying real algebraic varieties via effective complex techniques. 
Unlike other efforts in this direction (see for instance \cite{alberti2009isotopic} which stratifies hypersurfaces in $\RR^3$), the algorithm described here works in general for all dimensions. Moreover, the same techniques also readily adapt to the more general setting of {\em full semialgebraic sets} (see Definition \ref{def:full}).

Here we also investigate stratifications of algebraic maps $f:X \to Y$ between real varieties. This amounts to a pair of Whitney stratifications $X_\bullet$ of $X$ and $Y_\bullet$ of $Y$ such that $f$ sends each stratum of $X_\bullet$ submersively to a stratum of $Y_\bullet$. If $f$ is a {\em proper map} (this always holds if $X$ is compact, for instance), then it forms a {\em stratified fiber bundle} with respect to $X_\bullet$ and $Y_\bullet$. In other words, the restriction $f^{-1}(N) \surj N$ is an ordinary fiber bundle for each stratum $N \subset Y$. As a consequence, the homeomorphism type of the fiber $f^{-1}(y)$ depends only on the stratum $N$ and is independent of the choice of point $y \in N$. To compute such a stratification of $f$, we once again proceed by first treating it as a complex algebraic map $f_\CC:X(\CC) \to Y(\CC)$ and then using algorithms from \cite{hn}. We then show in Theorem \ref{thm:realMapsStrat} that the resulting stratifications of $X(\CC)$ and $Y(\CC)$ consist of varieties defined by real polynomials, and that the corresponding real varieties furnish the desired stratification of $f$. 

 Unfortunately, the requirement that $f:X \to Y$ be proper fails quite often in cases of interest for general algebraic maps. One exception to this rule is when $f$ is a {\em dominant morphism} between varieties of the same dimension (see Definition \ref{def:dominant}). For such a morphism $f$, we are able to partially recreate the stratified fiber bundle property in Theorem \ref{thm:dominantstrat} by carefully analysing and decomposing the locus of points at which $f$ fails to be proper. 

\subsection*{Roadmap} 
The rest of this paper is organised as follows.
In Section \ref{sec:complex} we summarise relevant results and algorithms from \cite{hn} for stratifying complex varieties. The only new material here lies in Section \ref{subsec:StratMapsVaritieyCC}, where we are able to remove a genericity hypothesis for map stratifications. Section \ref{sec:polarStrat} introduces a new {\bf WhitneyPolar} algorithm for complex Whitney stratifications which is based on polar varieties. In Section \ref{sec:realvarietystrat} we show how both our old and new complex stratification algorithms may be used to stratify real varieties and algebraic morphisms between them. Our task in Section \ref{sec:semialgstrat} is to similarly stratify basic closed semialgebraic sets whose inequality loci are open subsets of the ambient space --- these often arise in practical contexts owing to boundedness or positivity constraints. Finally, in Section \ref{sec:performance}, we examine the 
experimental performance of {\bf WhitneyPolar}.

\subsection*{Acknowledgements}
{\footnotesize The material in this article is based upon work supported by the Air Force Office of Scientific Research (AFOSR) under award
number FA9550-22-1-0462, managed by Dr.~Frederick Leve;  both MH and VN were partially supported by this award and would like to gratefully acknowledge this support. MH is also supported by the Royal Society under grant RSWF\textbackslash R2\textbackslash 242006 and would like to gratefully acknowledge this support.
AL is grateful to be partially supported by NSF DMS award 2001267 and the Simons Fellows program.
Finally, we thank the two anonymous referees whose suggestions greatly improved our paper.}
 
\section{Whitney Stratifications from Conormal Ideals}\label{sec:complex}

In this section we begin with a brief review of the stratification algorithms for complex varieties and their morphisms. Section \ref{subsec:stratificationReview} and Section \ref{subsec:flagStrat} are a review of the results of \cite{hn}, restated in a form better suited to the current presentation. On the other hand, Section \ref{subsec:StratMapsVaritieyCC} is mostly new and makes a key generalization of the map stratification algorithm of \cite[Section 6]{hn}; it allows us to drop the genericity requirement from \cite[Definition 6.3]{hn}, so that the Thom-Boardman flag of the given polynomial map $f:X\to Y$ need not intersect $X$ transversely\footnote{See \cite[Chapter VI]{golubitsky} for a detailed treatment of Thom-Boardman flags.}. Instead, we consider arbitrary polynomial maps $f:X\to Y$ between complex affine varieties $X\subset \CC^n$ and $Y\subset \CC^m$, and derive (in Lemma \ref{lemma:TBX}) an expression for the locus of smooth points of $X$ where the rank of $f$'s Jacobian drops below its generic value. 

A central role in our generalisation is played by a new type of stratification, which we have called {\em Whitney-Thom-Boardman} stratifications of $X$ along $f$ (see Definition \ref{def:WTBStrat}). Here the decomposition of $X$ into strata is well behaved with respect to the Thom-Boardman singularities of $f$. 

\subsection{Stratification of Complex Varieties}\label{subsec:stratificationReview}
Consider a pure $k$-dimensional complex variety, that is a variety all of whose irreducible components are of dimension $k$, $X=\bV_\CC(I_X)\subset \CC^n$ defined by a radical ideal $I_X=\langle f_1, \dots, f_r\rangle$ in  $\CC[x_1,\dots, x_n]$. Let $X_{\rm reg}$ denote the (open, dense) manifold of smooth points in $X$ and let $X_{\rm sing}=X-X_{\rm reg}$ be the singular locus of $X$; recall $X_{\rm sing}$ is a closed proper subvariety of $X$, see e.g.~\cite[Section 6.1]{smith2004invitation} or ~\cite[page 26--27]{michalek2021invitation}.  The {\em conormal variety} of $X$ is the subvariety of $X
 \times \pp^{n-1}$ given by 
 \begin{align}\label{eq:conormal}
{\rm\bf Con}(X)=\overline{\left\lbrace  (p,\xi)\;|\; p \in X_{\rm reg} \text{ and } T_pX_{\rm reg} \subset\xi^\perp \right\rbrace}.
\end{align} In other words, ${\bf Con}(X)$ is obtained by closing the set of pairs $(p,\xi)$ consisting of points $p \in X_{\rm reg}$ and $\xi \in \pp^{n-1}$ such that the inner product $\ip{\xi,v}$ is zero for all $v \in T_pX_{\rm reg}$. We note the later coordinate is taken to be projective as we identify all constant multiples of a particular unit direction vector defining the normal to a plane. The conormal variety comes endowed with a canonical projection map $\kappa_X:{\rm\bf Con}(X)\to X$ induced by the coordinate projection $\CC^n\times \PP^{n-1}\to \CC^n$. Work in the ring $\CC[x_1,\dots,x_n][\xi_1,\dots, \xi_n]$ and set $$
\mathscr{K}=\begin{bmatrix}
\xi_1& \cdots & \xi_n\\
	\frac{\partial f_1}{\partial x_1} &\cdots& \frac{\partial f_1}{\partial x_n}\\
	\vdots& \ddots & \vdots \\
	\frac{\partial f_m}{\partial x_1} &\cdots& \frac{\partial f_m}{\partial x_n}\\
	\end{bmatrix} \quad \text{ and }
 \quad \mathscr{J}=\begin{bmatrix}
	\frac{\partial f_1}{\partial x_1} &\cdots& \frac{\partial f_1}{\partial x_n}\\
	\vdots& \ddots & \vdots \\
	\frac{\partial f_m}{\partial x_1} &\cdots& \frac{\partial f_m}{\partial x_n}\\
	\end{bmatrix}.
$$ The conormal variety ${\rm\bf Con}(X)=\bV_\CC(I_{{\rm\bf Con}(X)})$ is defined by the radical ideal obtained via the saturation $I_{{\rm\bf Con}(X)}=(I_X+K):J^\infty$ --- here $K$ is the ideal generated by all $(n-k+1)\times (n-k+1)$ minors of the matrix $\mathscr{K}$, while $J$ is similarly defined by the $(n-k)\times (n-k)$ minors of the matrix $\mathscr{J}$. Note that the singular locus $X_{\rm sing}$ is precisely $\bV_\CC(J) \cap X$. A definition of ideal saturation and a discussion of how it is computed may be found in books such as~\cite[\S4.4]{CLO}. The following result \cite[Theorem 3.1]{hn} forms the basis for efficient Whitney stratification of a given variety $X$. Recall that given an ideal $I$ its associated primes are the radicals of the ideals appearing in a minimal primary decomposition of $I$, see books such as \cite{CLO,michalek2021invitation} for definitions and see \cite{decker1999primary} for a survey of methods for the computation of associated primes of an ideal via Gr\"obner basis.

\begin{theorem} Let $X \subset \CC^n$ be a pure dimensional variety, let $Y\subset X_{\rm sing}=(X-X_{\rm reg})$ be a nonempty irreducible subvariety defined by a radical ideal $I_Y$, and set $I_{\kappa^{-1}_X(Y)}:=I_{\Con(X)}+I_Y$. 
Let $\{P_1,\dots, P_s\}$ be the associated primes of { $
I_{{\kappa_X^{-1}(Y)}},
$}
let { $\sigma \subset \set{1,2,\ldots,s}$} be the set of indices $i$ with {$\dim \kappa_X(\bV_\CC(P_i)) < \dim Y$} and let 
$$A := \left[\bigcup_{i \in \sigma}\kappa_X(\bV_\CC(P_i))\right] \cup Y_{\rm sing}.$$
Then the pair {$(X_\text{\rm reg},Y-A)$ satisfies Condition (B)}. \label{thm:WhitB_Primary_Decomp}
\end{theorem}

We first describe a subroutine, called {\bf Decompose}, which implements this theorem for a given subvariety $Y \subset X_\text{\rm sing}$. In particular, it computes the dimension of the elimination ideal arising from each primary component of $I_{\kappa_X^{-1}(Y)}$, and records those which have dimension less than $Y$. In effect, this subroutine explicitly constructs subvarieties of $Y$ (not necessarily contained in $Y_\text{\rm sing}$ where Condition (B) fails with respect to $X$.

\medskip
	\begin{center}
		\begin{tabular}{|r|l|}
			\hline
			~ & {\bf Decompose}$(Y,X )$ \\
			\hline
			~&{\bf Input:} Algebraic varieties $Y \subset X$ in $\CC^n$, with $d:=\dim Y$.\\
			~&{\bf Output:} A list of subvarieties $Y_\bullet$ of $Y$. \\
			\hline
			1 & {\bf Set} $Y_\bullet := (Y_d,Y_{d-1},\ldots,Y_0) := (Y,\varnothing,\ldots,\varnothing)$ \\
			2 &  {\bf Set} $J:=I_{\Con(X)}+I_{Y}$ \\
			3 & {\bf For each} primary component $Q$ of a primary decomposition of $J$\\
			4 & \spc {\bf Set} $K := Q\cap \CC[x]$ \\
			5 & \spc {\bf If}  $\dim \bV_\CC(K) < \dim Y$   \\
			6 & \spc \spc{\bf Add} $\bV_\CC(K)$ to $Y_{ \geq\dim \bV_\CC(K)}$\\
			7 & {\bf Return} $Y_\bullet$ \\
			\hline
		\end{tabular} 
	\end{center}\medskip

It is shown in \cite{hn} that the {\bf WhitStrat} algorithm below correctly computes the Whitney stratification of the variety $X$, and is guaranteed to terminate in finitely many steps. 

\medskip
	\begin{center}
		\begin{tabular}{|r|l|}
			\hline
			~ & {\bf WhitStrat}$(X)$ \\
			\hline
			~&{\bf Input:} A pure $k$-dimensional variety $X\subset \CC^n$.\\
			~&{\bf Output:} A list of subvarieties $X_\bullet$ of $X$. \\
			\hline
			1 & {\bf Set} $X_\bullet := (X_k,X_{k-1},\ldots,X_0) := (X,\varnothing,\ldots,\varnothing)$ \\
			2 & {\bf Compute} $X_{\rm sing}$ and $\mu := \dim(X_{\rm sing})$ \\
			3 & {\bf For each} irreducible component $Z$ of $X_{\rm sing}$  \\
			4 & \spc {\bf Add} $Z$ to $X_{\geq \dim Z}$ \\
			5 & {\bf For each} $d$ in $(\mu, \mu-1,\ldots,1,0)$ \\
			6 & \spc {\bf Set }$X_\bullet := {\bf Merge} (X_\bullet,{\bf Decompose}(\Pure_d(X_d),X))$ \\
			7 & \spc {\bf Set }$X_\bullet := {\bf Merge} (X_\bullet,{\bf WhitStrat}(\Pure_d({X_d})))$\\
			8 & {\bf Return} $X_\bullet$ \\
			\hline
		\end{tabular}
	\end{center}
\medskip

 \noindent The {\bf WhitStrat} function makes use of two elementary subroutines besides {\bf Decompose}; these are invoked in lines 6 and 7: 
 \begin{enumerate} 
 \item $\Pure_d$ extracts the purely $d$-dimensional irreducible components of the input variety; this is accomplished via prime decomposition \cite{decker1999primary,eisenbud1992direct}.
 \item Let $V_\bullet$ and $W_\bullet$ be nested sequences of subvarieties of a common variety, with the length $p$ of $V_\bullet$ larger than the length $q$ of $W_\bullet$. Then {\bf Merge}$(V_\bullet,W_\bullet)$ creates a new sequence $U_\bullet$ of length $p$ via $U_i := V_i \cup W_{\max(i,q)}$. 
 \end{enumerate}
 
\subsection{Flag-Subordinate Stratifications}\label{subsec:flagStrat}

By a {\em flag} $\bF_\bullet$ on a variety $X\subset \CC^n$ we mean any finite nested set of subvarieties of the form
	\[
	\varnothing = \bF_{-1}X \subset \bF_0X \subset \bF_1X \subset \cdots \subset \bF_{\ell-1}X \subset \bF_\ell X = X.
	\] The integer $\ell$  is called the {\em length} of  $\bF_\bullet$. There are no additional restrictions on the dimensions of the individual $\bF_iX$ and, in particular, we do not require successive differences $\bF_iX - \bF_{i-1}X$ to be smooth manifolds or to satisfy Condition (B).
	
	\begin{definition}\label{def:flagsubstrat}
		Let $X \subset \CC^n$ be an affine variety and $\bF_\bullet$ a flag on $X$ of length $\ell$. A Whitney stratification $X_\bullet$ of $X$ is {\bf subordinate} to $\bF_\bullet$ if for each stratum $S \subset X$ of $X_\bullet$ there exists some $j = j(S)$ in $\set{0,\ldots,\ell}$ satisfying $S \subset (\bF_jX - \bF_{j-1}X)$.
	\end{definition}
 Our main motivation for considering flag subordinate stratification is to develop algorithms to stratify algebraic maps between varieties. It is shown in \cite[\S5]{hn} that a flag subordinate stratification of a complex algebraic variety can always be  computed using the {\bf WhitStratFlag} algorithm described below. There are two main subroutines, the first of which is called {\bf InducedFlag}. This takes as input a subvariety $W \subset X$ and a flag $\bF_\bullet$ on $X$, and outputs the restriction of $\bF_\bullet$ to $W$.

\medskip

	\begin{center}
		\begin{tabular}{|r|l|}
			\hline
			~ & {\bf InducedFlag}$(W,\bF_\bullet)$ \\
			\hline
			~&{\bf Input:} A subvariety $W \subset X\subset \CC^n$ and a flag $\bF_\bullet$ on $X$ of length $\ell$.\\
			~&{\bf Output:} A flag $\bF'_\bullet$ on $W$ of length $\ell$. \\
			\hline
			1 & {\bf Set} $\bF'_\bullet W := (\bF'_\ell W, \ldots, \bF'_0W) := (\varnothing,\ldots,\varnothing)$\\
			2 & {\bf For each} irreducible component $V$ of $W$  \\
			3 & \spc {\bf Add} $V$ to $\bF'_{i}W$ {\bf for all} $\bF'_{i}$ where $V\subset \bF'_{i}$\\
			4 & \spc {\bf For each} $j$ with $\dim(\bF_jX \cap V) < \dim V$ \\
			5 & \spc \spc {\bf Add} $V_j := (\bF_jX \cap V)$ to $\bF'_{i} W$ {\bf for all} $\bF'_{i}$ where $V_j\subset \bF'_{i}$ \\
			6 & {\bf Return} $\bF'_\bullet W$\\
			\hline
		\end{tabular}
	\end{center}
\medskip

The second subroutine is called {\bf DecomposeFlag}; this is a variant of the ${\bf Decompose}$ subroutine described in the preceding subsection. The only difference is that rather than merging the detected subvariety $\bV_\CC(K)$ with $Y_\bullet$ directly, we first use {\bf InducedFlag} to restrict $\bF_\bullet$ to $\bV_\CC(K)$ and then merge the output with $Y_\bullet$. 

\medskip
 \begin{center}
		\begin{tabular}{|r|l|}
			\hline
			~ & {\bf DecomposeFlag}$(Y,X,\bF_\bullet)$ \\
			\hline
			~&{\bf Input:} Varieties $Y \subset X\subset \CC^n$ with $d:=\dim Y$ and a flag $\bF_\bullet$ on $X$.\\
			~&{\bf Output:} A list of subvarieties $Y_\bullet \subset Y$. \\
			\hline
			1 & {\bf Set} $Y_\bullet := (Y_d,Y_{d-1},\ldots,Y_0) := (\varnothing,\ldots,\varnothing)$ \\
			2 &  {\bf Set} $J:=I_{\Con(X)}+I_{Y}$ \\
			3 & {\bf For each} primary component $Q$ of a primary decomposition of $J$\\
			4 & \spc {\bf Set} $K := Q\cap \CC[x]$ \\
			5 & \spc {\bf If}  $\dim \bV_\CC(K) < \dim Y$   \\
			6 & \spc \spc {\bf Merge} $Y_\bullet$ with {\bf InducedFlag}$(\bV_\CC(K),\bF_\bullet X)$\\
			7 & {\bf Return} $Y_\bullet$ \\
			\hline
		\end{tabular}
	\end{center}

 \medskip

Finally, here is the promised {\bf WhitStratFlag} algorithm which makes use of these two subroutines. This takes as input a variety $X$ and a flag $\bF_\bullet$ defined on $X$. The output $X_\bullet$ is guraranteed to be a Whitney stratification subordinate to $\bF_\bullet$ in the sense of Definition \ref{def:flagsubstrat} --- see \cite[Sec 5]{hn} for details.

\begin{center}
		\begin{tabular}{|r|l|}
			\hline
			~ & {\bf WhitStratFlag}$(X,\bF_\bullet)$ \\
			\hline
			~&{\bf Input:} A pure $k$-dimensional variety $X\subset \CC^n$ and a flag $F_\bullet$ on $X$.\\
			~&{\bf Output:} A list of subvarieties $X_\bullet \subset X$. \\
			\hline
			1 & {\bf Set} $X_\bullet := (X_k,X_{k-1},\ldots,X_0) := (X,\varnothing,\ldots,\varnothing)$ \\
			2 & {\bf Compute} $X_{\rm sing}$ and $\mu := \dim(X_{\rm sing})$ \\
			3 & {\bf Set} $X_d = X_{\rm sing}$ for all $d$ in $\set{\mu,\mu+1,\ldots,k-1}$\\
			4 & {\bf Merge} $X_\bullet$ with {\bf InducedFlag}$(X_{\rm sing},\bF_\bullet X)$\\
			5 & {\bf For each} $d$ in $(\mu, \mu-1,\ldots,1,0)$ \\
			6 & \spc {\bf Merge} $X_\bullet$ with ${\bf DecomposeFlag}(X_d,X,\bF_\bullet)$ \\
			7 & \spc {\bf Merge} $X_\bullet$ with ${\bf WhitStratFlag}(X_d,\bF_\bullet)$\\
			8 & {\bf Return} $X_\bullet$ \\
			\hline
		\end{tabular}
	\end{center}
\medskip

\subsection{Stratifying Complex Algebraic Morphisms}\label{subsec:StratMapsVaritieyCC}

Maps of Whitney stratified spaces are typically required to satisfy additional criteria beyond smoothly sending strata to strata --- see \cite[Def 3.5.1]{brasseletSeadeSuwa} or \cite[Part I, Ch 1.7]{SMTbook} for instance. 

\begin{definition}\label{def:stratmap3}
 Let $\mathscr{X}_\bullet$ and $\mathscr{Y}_\bullet$ be Whitney stratifications of topological spaces $\mathscr{X}$ and $\mathscr{Y}$. A continuous function $\phi:\mathscr{X} \to \mathscr{Y}$ is {\bf stratified} with respect to $\mathscr{X}_\bullet$ and $\mathscr{Y}_\bullet$ if for each stratum $M \subset \mathscr{X}$ there exists a stratum $N \subset \mathscr{Y}$ satisfying two requirements: 
		\begin{enumerate}
			\item the image $\phi(M)$ is wholly contained in $N$; and moreover,
		\item the restricted map $\phi|_M:M \to N$ is a smooth submersion.\footnote{Explicitly, its derivative $d(\phi|_S)_x:T_xM \to T_{\phi(x)}N$ is surjective at each point $x$ in $M$.}
		\end{enumerate}
  The pair $(\mathscr{X}_\bullet,\mathscr{Y}_\bullet)$ is called a stratification of $\phi$.
	\end{definition}
 
\begin{remark} The second requirement of Definition \ref{def:stratmap3} ensures the following crucial property via Thom's first isotopy lemma \cite[Prop 11.1]{Mather2012}. If $\phi$ is a {\em proper map} -- namely, if the inverse image of every compact subset of $Y$ is compact in $X$ -- then for every stratum $N \subset Y$, the restriction of $\phi$ forms a locally trivial fiber bundle from $\phi^{-1}(N)$ to $N$. \label{remark:fibration}
\end{remark}

Fix varieties $X \subset \CC^n$ and $Y \subset \CC^m$ along with a polynomial map $f:\CC^n \to \CC^m$ that sends points of $X$ to points of $Y$. In \cite[Section 6]{hn}, we described an effective algorithm for stratifying any $f:X \to Y$ whose {\em Thom-Boardman flag}\footnote{See \cite[Chapter VI]{golubitsky}.} intersects $X$ transversely. Here we will describe a new modification which circumvents this transversality requirement. The first step towards this goal is the following result. 

\begin{lemma}\label{lemma:TBX}
 Let $G := \set{g_1,g_2,\ldots,g_s}$ be any set of polynomials defining $X \subset \CC^n$ which generate a radical ideal. Consider, at each $x \in X$, the matrix of partial derivatives
 \[
 \mathbf{M}_G(x) :=	\begin{bmatrix}
{\partial g_1}/{\partial x_1} &\cdots& {\partial g_s}/{\partial x_1}\\
 \vdots& \ddots & \vdots  \\
	{\partial g_1}/{\partial x_n} &\cdots& {\partial g_s}/{\partial x_n}\\
	\end{bmatrix}
\]  
evaluated at $x$. For each $j$ in $\set{0,1,\ldots,\dim X}$, let $V_j = V_j(X;f)$ be the subvariety of $\CC^n$ generated by all $j \times j$ minors of the augmented matrix $\mathbf{A}(x) := \begin{bmatrix} \mathbf{M}_G(x) & \jac f(x)^\text{\sf T}\end{bmatrix}$. Writing $f^*:X_\text{\rm reg} \to \CC^m$ for the restriction of $f$ to the smooth points of $X$, we have
\[
\text{\rm rank } \jac f^*(x) \leq i \quad \text{ if and only if } \quad x \in V_{(n-\dim X) + (i+1)},
\]
for all $i$ in $\set{0,1,\ldots,\min(\dim X,m)}$.
\end{lemma}
\begin{proof}
At each point $x \in X_{\rm reg}$, the left block of $\mathbf{A}(x)$ equals $\mathbf{M}_G(x)$, whose $s$ columns span the orthogonal complement of $T_xX_{\rm reg}$ in $\CC^n$. Since this complement necessarily has dimension $(n-\dim X)$ regardless of the chosen $x$, we obtain
\[
\text{rank }\textbf{A}(x) \geq n - \dim X.
\]
The right block, consisting of $\jac f(x)^\text{\sf T}$, has columns which span the coimage of $\jac f(x)$, i.e., the orthogonal complement of $\ker \jac f(x)$ in $\CC^n$. Thus, the following criteria are equivalent:
\begin{enumerate}
    \item $\text{rank }\jac f^*(x) = 0$,
    \item the coimage of $\jac f(x)$ is a subspace of the orthogonal complement $T_xX_{\rm reg}^\perp$,
    \item the columns of $\jac f(x)^\mathsf{T}$ are spanned by those of $\mathbf{M}_G(x)$,
    \item $\mathbf{A}(x)$ assumes its minimum possible rank of $n - \dim X$.
\end{enumerate}  
Similarly, $\text{rank }\jac f^*(x) = i$ holds whenever $\text{rank }\mathbf{A}(x) = (n - \dim X) + i$; this rank condition on $\mathbf{A}(x)$ corresponds to the vanishing of all minors of size $(n - \dim X) + i+1$, as claimed above.
\end{proof}

In light of the preceding result, we will work with the {\bf rank flags} of $f$ along various subvarieties of $X$; given such a subvariety $Z \subset X$, the relevant flag is defined by
\begin{align}
\B_iZ := Z \cap V_{n-\dim Z+i+1}(Z;f).
\end{align} Here $V_\bullet(Z;f)$ is given in the statement of Lemma \ref{lemma:TBX}, with $0 \leq i \leq \min(\dim Z,m)$. By construction, the Jacobian of the restriction $f^*:Z_{\rm reg} \to \CC^m$ has rank $\leq i$ at a point $z \in Z_{\rm reg}$ if and only if $z$ lies in $\B_iZ$.
 
\begin{definition}\label{def:WTBStrat}
A {\bf Whitney-Thom-Boardman} (WTB) stratification of $X$ along the map $f:\CC^n \to \CC^m$ is any pair $(X_\bullet, \bS_\bullet X)$, where $X_\bullet$ is a Whitney stratification of $X$ while $\bS_\bullet X$ is a flag on $X$ such that two conditions hold:
\begin{enumerate}
\item $X_\bullet$ is subordinate to $\bS_\bullet X$, and 
\item for each $X_\bullet$-stratum $M \subset X$, we have $\B_i\overline{M} \subset \bS_iX$ whenever $0 \leq i \leq \dim M$.
\end{enumerate}
\end{definition}

The first requirement of this definition forces $X_\bullet$ to depend on $\bS_\bullet X$ whereas the second requirement makes $\bS_\bullet X$ dependant on $X_\bullet$. We are therefore compelled to construct both $X_\bullet$ and $\bS_\bullet X$ simultaneously via the algorithm described below. It will also be necessary in the sequel to produce WTB stratifications $(X_\bullet,\bS_\bullet X)$ for which $X_\bullet$ is subordinate to an auxiliary flag $\bF_\bullet X$. Setting $\bF_iX=X$ for all $i$, we obtain a WTB stratification as defined above. To simplify the presentation, we will denote this important special case as {\bf WTBStrat}$(X,f)$. \medskip

	\begin{center}
		\begin{tabular}{|r|l|}
			\hline
			~ & {\bf WTBStratFlag}$(X,f,\bF_\bullet X)$ \\
			\hline
			~&{\bf Input:} $f: X \to \CC^m$ algebraic map with $k := \dim X$;  $F_\bullet X$ a flag on $X$.\\
			~&{\bf Output:} A pair of flags $(X_\bullet,\bS_\bullet X)$ on $X$. \\
			\hline
			1 & {\bf Set} $X_\bullet := (X_k,X_{k-1},\ldots,X_0) := (X,\varnothing,\ldots,\varnothing)$ \\
      2 & {\bf Compute} $\B_\bullet X$ using Lemma \ref{lemma:TBX} \\
   3 & {\bf Set} $\bS_\bullet X := \B_\bullet X$\\
			4 & {\bf Compute} 
   $X_{\rm sing}$ and $\mu := \dim(X_{\rm sing})$ \\
   5 & {\bf Compute} $\B_\bullet X_{\rm sing}$ using Lemma \ref{lemma:TBX} \\
			6 & {\bf Set} $X_d := X_{\rm sing}$ for all $d$ in $\set{\mu,\mu+1,\ldots,k-1}$\\
   7 & {\bf Set} $\bS_\bullet X := {\bf Merge}(\bF_\bullet X,\B_\bullet X_{\rm sing})$\\
			8 & {\bf Set} $X_\bullet := {\bf Merge} (X_\bullet,{\bf InducedFlag}(X_{\rm sing},\bS_\bullet X))$\\
			9 & {\bf For each} $d$ in $(\mu, \mu-1,\ldots,1,0)$ \\
			10 & \spc {\bf Set }$X_\bullet := {\bf Merge}(X_\bullet,{\bf DecomposeFlag}(X_d,X,\bS_\bullet X))$ \\
			11 & \spc {\bf Set }$X_\bullet := {\bf Merge}(X_\bullet,{\bf WTBStratFlag}(X_d,f,\bF_\bullet X))$\\
      12 & \spc {\bf Compute} $\B_\bullet X_{d}$ using Lemma \ref{lemma:TBX} \\
    13 &\spc {\bf Set} $\bS_\bullet X := {\bf Merge}(\bS_\bullet X,\B_\bullet X_{d})$\\
			14 & {\bf Return} $(X_\bullet,\bS_\bullet X) $ \\
   
			\hline
		\end{tabular}
	\end{center}\medskip

 \begin{theorem} Let $X\subset \CC^n$ be a pure dimensional complex algebraic variety, $f:\CC^n\to \CC^m$ a polynomial map, and $\bF_\bullet X$ a flag on $X$. Then, 
 \begin{enumerate}
     \item ${\bf WTBStratFlag}(X,f,\bF_\bullet X)$ terminates,
     \item its output $(X_\bullet,\bS_\bullet X)$ is a WTB stratification of $X$ along $f$, and
     \item the Whitney stratification $X_\bullet$ is subordinate to $\bF_\bullet X$.
 \end{enumerate}  \label{thm:Stat_wrt_Flag}
	\end{theorem}
	\begin{proof}
		The termination of {\bf WTBStratFlag} follows from noting that the recursive call in Line 9 involves the variety $X_d$ of dimension at most $d$, with $d$ strictly smaller than $k = \dim X$. And the fact that $X_\bullet$ is a Whitney stratification of $X$ follows from the correctness of the original {\bf WhitStrat} algorithm \cite[Sec 4.4]{hn}. Moreover, since $\bF_\bullet X$ is merged into $\bS_\bullet X$ via Line 7, verifying that $X_\bullet$ is $\bF_\bullet$-subordinate can be reduced to checking that it is $\bS_\bullet$-subordinate. Therefore, it suffices to confirm that the output pair $(X_\bullet,\bS_\bullet X)$ satisfies the two requirements of Definition \ref{def:WTBStrat}. To this end, consider an $i$-dimensional stratum $M$ of $X_\bullet$.
  
        To verify requirement (1), note that $M$ may be uniquely written as $Z - X_{i-1}$, where $Z \subset X_i$ is an irreducible component --- in particular, this $Z$ will appear as a $V$ in Line 3 of  {\bf InducedFlag} when it is called with first input $X_i$. Let $j$ be the minimal index satisfying $Z \subset \bS_jX$. Then $\dim(Z \cap \bS_\ell X) < i$ for every $\ell < j$, and so $Z_\ell=Z \cap \bS_\ell X$ is merged with $X_\bullet$ in Line 5 of the {\bf InducedFlag}. Thus, $Z_\ell$ is contained in $X_m$ for some $m < i$. Since $X_m \subset X_i$, it follows that 
		\[
		M \cap (\bS_{p}X - \bS_{p-1}X) = \varnothing \text{ whenever } p<j.
		\] But since $M \subset \bS_{j}X$ and $Z$ is irreducible, we have $Z \subset \bS_{j}X$, whence $Z\subset \bS_{p}X$ for all $p\geq j$. Thus, $M$ also has empty intersections with $\bS_{p}X -\bS_{p-1}X$ for $p > j$, as desired. Turning now to property (2) of Definition \ref{def:WTBStrat}, we observe that $\bS_\bullet X$ is initialized to $\B_\bullet X$ in line 3 of {\bf WTBStrat}. Subsequently, $\B_\bullet X_d$ is merged into $\bS_\bullet X$ for each $d \leq \dim X_{\rm sing}$. Thus, for every $i$-stratum $M \subset X$, we have enforced that $\B_j \overline{M} \subset \B_j X_i \subset \bS_j X$ whenever $j \leq i$. 
	\end{proof}

Given a WTB stratification $(X_\bullet,\bS_\bullet X)$ of $X$ along $f$, we construct a {\em pushforward} flag $\bB_\bullet$ on the codomain variety $Y$ as follows. Writing $\ell$ for the length of $\bS_\bullet$, define
\begin{align} \label{eq:Bdef}
\bB_iY := \begin{cases} 
            f(\bS_iX) & i \leq \ell \\
            Y & i = \ell + 1
           \end{cases}.
\end{align} 
The map stratification algorithm described below first constructs a $\bB_\bullet$-subordinate stratification $Y'_\bullet$ of $Y$. It then constructs a WTB stratification $(X''_\bullet,\bS'_\bullet X)$ of $X$ along $f$ such that $X''_\bullet$ is subordinate to the flag $\bF_iX := f^{-1}(Y_i)$.\medskip
 
	\begin{center}
		\begin{tabular}{|r|l|}
			\hline
			~ & {\bf WhitStratMap}$(X,Y,f)$ \\
			\hline
			~&{\bf Input:} Pure dimensional varieties $X,Y$ and 
			a any morphism $f:X \to Y$.\\
			~&{\bf Output:} Lists of subvarieties $X_\bullet \subset X$ and $Y_\bullet \subset Y$. \\
			\hline
			~1 & {\bf Set} $(X''_\bullet,{\bS''_\bullet}X) := {\bf WTBStrat}(X,f)$ \\ 
			~2 & {\bf Set} ${\bB'_\bullet}Y := (\bB'_{k+2}Y,\ldots,\bB'_0Y) := (Y,\varnothing,\ldots,\varnothing)$ \\
			~3 & {\bf For each} $j$ in $(0,1,\ldots,k)$ \\
			~4 & \spc {\bf Set } $\bB'_jY := f(\bS''_jX)$\\
			~5 & {\bf Set} $\bB'_{k+1}Y := Y$ \\
			~6 & {\bf Set} $Y'_\bullet := {\bf WhitStratFlag}(Y,\bB'_\bullet)$ \\
			~7 & {\bf For each} $i$ in $(0,1,\ldots,\dim Y)$ \\
			~8 & \spc {\bf Set} $\bF_iX := f^{-1}(Y'_i)$ \\
			9 & {\bf Set} $(X'_\bullet, {\bS'_\bullet}X):=  {\bf WTBStratFlag} (X,f,\bF_\bullet)$ \\
			10 & {\bf Set} $(X_\bullet,Y_\bullet) := ${ \bf  Refine}$(X'_\bullet,Y'_\bullet,f)$\\	
			11 & {\bf Return} $(X_\bullet,Y_\bullet)$\\
			\hline
		\end{tabular}
	\end{center}\medskip
The final line invokes a {\bf Refine} subroutine, which will be described later. We first highlight a crucial property of the Whitney stratifications $(X'_\bullet,Y'_\bullet)$ which are invoked in Line 10.

\begin{proposition}\label{propn:JacobianRankRestrictionMap}
   For every $X'_\bullet$-stratum $M \subset X$ there exist a $Y'_\bullet$-stratum $N \subset Y$ satisfying $f(M) \subset N$. Moreover, at each point $x\in M, $ the Jacobian $\jac f|_M(x): T_xM \to T_{f(x)}N$ of the restricted map $f|_M$ has maximal rank: $${\rm rank}(\jac f|_M(x))=\min(\dim M,\dim N).$$ 
\end{proposition}\begin{proof}
Noting that $X'_\bullet$ is subordinate to the flag $\bF_\bullet$ by Line 9  of {\bf WhitStratMap}, we know that for each stratum $M$ of $X'_\bullet$ there is a number $i := i(M)$ satisfying $M \subset (\bF_iX-\bF_{i-1}X)$. Now by Line 8, for any such stratum we have $f(M) \subset (Y'_i - Y'_{i-1})$. By the definition of a Whitney stratum we know that $M$ is connected, and so its image under the continuous map $f$ must also be connected; thus there is a unique stratum $N \subset (Y'_i - Y'_{i-1})$ of $Y'_\bullet$ satisfying  $f(M) \subset N$, as claimed above. It remains to establish that the Jacobian of $f|_M:M \to N$ has maximal rank at every point of $x$.

By Line 6 of {\bf WhitStratMap}, the stratification $Y'_\bullet$ is subordinate to the flag $\bB'_\bullet Y$; thus, there exists a number $j := j(N)$ satisfying $N \subset (\bB'_jY - \bB'_{j-1}Y)$. From Line 9, we note that $\bB'_\bullet$ is the pushforward of $\bS''_\bullet$, so we obtain 
\begin{align}\label{eq:finv}
f^{-1}(N) \subset \bS''_jX - \bS''_{j-1}X.
\end{align} 
Recall that $\bS''_jX$ contains $\B_j\overline{M}$ by the second requirement of \ref{def:WTBStrat} (and similarly for $j-1$). Therefore, the restricted map $f^*:\overline{M}_{\rm reg} \to \CC^m$ has Jacobian of constant rank $\rho_j$ given by 
\begin{align}\label{eq:rho}
\rho_j := \min(\dim M, m)-j.
\end{align} 
Since $M$ is an open subset of $\overline{M}_{\rm reg}$, for each $x \in M$ we immediately have $
\text{rank }\big(\jac f^* (x)\big) = \rho_j.$ 
Let $\iota:N \inj \CC^m$ be the inclusion map, and recall that $f|_M:M \to N$ denotes the restriction of $f$ to $M$, now viewed as a smooth map to $N$ rather than  to $\CC^m$. Thus, $f^* = \iota \circ f|_M$, and the following diagram of derivatives commutes by the chain rule:
\[
\xymatrixcolsep{1in}
\xymatrixrowsep{.5in}
\xymatrix{
T_xM \ar@{->}[r]^{\jac f|_M(x)} \ar@{->}[dr]_{\jac f^*(x)}& T_{f(x)}N \ar@{->}[d]^{\jac\iota({f(x)})}\\
& \CC^m
}
\]
Since $\iota$ is an embedding, its Jacobian is injective, so in particular the rank of $\jac f|_M(x)$ also equals $\rho_j$. Thus, we immediately have $\rho_j \leq \min(\dim M, \dim N)$ and it only remains to establish the opposite inequality. There are two cases to consider --- 
\begin{enumerate} 
\item {\bf if $\dim M \leq \dim N$}: Note that $\dim N \leq m$ holds by the fact that $N$ is a submanifold of $\CC^m$. Thus, we have $\dim M \leq \dim N \leq m$, and so \eqref{eq:rho} gives $\rho_j = \dim M-j$ for some $j \geq 0$, whence $\rho_j \geq \dim M$.

\item {\bf if $\dim M > \dim N$}: Since $\bB'_\bullet$ is the pushforward of $\bS''_\bullet$, we have $\bB_kY = f(\bS''_kX)$ for all $k$. And since $\bS''_\bullet X$ is part of a WTB stratification of $X$ along $f$ by Line 1, we know from Definition \ref{def:WTBStrat} that the Jacobian of $f$ has constant rank $\rho_k$ on (the smooth part of) the difference $\Delta_k := \bS''_kX-\bS''_{k-1}X$. By the implicit function theorem, we therefore have $\rho_k = \dim \bB_kY$ at every $k$ for which $\Delta_k \neq \varnothing$. Since the nonempty set $f^{-1}(N)$ lies in $\Delta_j$ by \eqref{eq:finv}, we therefore have $\dim \bB'_jY = \rho_j$. Finally, since $N$ is an open subset of $f(\Delta_j) = \bB'_jY - \bB'_{j-1}Y$, we obtain $\rho_j = \dim N$.
\end{enumerate}
Thus, in both cases, $\rho_j$ equals $\min(\dim M, \dim N)$, and $\jac f|_M(x)$ has maximal rank.
\end{proof}

The stratifications $X'_\bullet$ and $Y'_\bullet$ obtained after the conclusion of Line 9 of {\bf WhitStratMap} do not generally satisfy the requirements of Definition \ref{def:stratmap3}. In particular, the crucial Jacobian-surjectivity requirement fails whenever $\dim M < \dim N$. As established above, in this case the Jacobian is injective rather than surjective. To fix this defect, we consider the set of problematic strata-pairs $\cP=\cP(X'_\bullet,Y'_\bullet)$ given by:
	\[
	\cP := \set{(M,N) \mid f(M) \subset N \text{ with } \dim M < \dim N}.
	\]
The {\bf Refine} subroutine (which is invoked in Line 10 of {\bf WhitStratMap}) fixes this defect. We note that a similar subroutine also appears in \cite{hn}; the main difference between the two {\bf Refine} subroutines is that the new one calls {\bf WTBStrat} instead of {\bf WhitStrat}.
\medskip

	\begin{center}
		\begin{tabular}	{|r|l|}
			\hline
			~ & {\bf Refine}$(X'_\bullet,Y'_\bullet,f)$ \\
			\hline
			~&{\bf Input:} Stratifications $X'_\bullet,Y'_\bullet$ of pure dimensional varieties $X$ and $Y$ \\ 
			& and a morphism $f:X \to Y$. \\
			~&{\bf Output:} Lists of subvarieties $X_\bullet \subset X$ and $Y_\bullet \subset Y$. \\
			\hline
			~1 & {\bf For each } $(S,R) \in \cP(X'_\bullet,Y'_\bullet)$ with $\dim R$ maximal \\
			~2 & \spc {\bf Set} $Y^+_\bullet := Y'_\bullet$ \\
			~3 & \spc {\bf Set} $d := \dim \overline{f(S)}$ \\
			~4 & \spc {\bf Add} $\overline{f(S)}$ to $Y^+_{\geq d}$ \\
			~5 &\spc {\bf Merge} $Y^+_\bullet$ with {\bf WhitStrat}$(\Pure_d(Y'_d))$\\
			~6 & \spc {\bf For each} $\ell = (d,d-1,\ldots,1,0)$ \\
			~7 & \spc \spc {\bf For each } irreducible $W \subset \overline{Y^+_\ell - Y'_\ell}$ and $S' \in \mathfrak{S}_\ell(X'_\bullet)$\\
			~8 &\spc \spc  \spc {\bf If } $Z \cap S' \neq \varnothing$ for an irreducible $Z \subset f^{-1}(W)$   \\
			~9 &  \spc  \spc  \spc \spc {\bf Set} $r := \dim Z$ \\
			10 &\spc  \spc  \spc\spc  {\bf Add} $Z$ to $X'_{\geq r}$\\
			11 &\spc  \spc  \spc\spc {\bf Merge} $X'_\bullet$ with {\bf WTBStrat}$({\rm Pure}_{r}(X'_{r}),f)$\\
			12 & \spc {\bf Set} $Y'_\bullet=Y^+_\bullet$\\
			13 & \spc {\bf Recompute} $\cP(X'_\bullet,Y'_\bullet)$\\
			14 & {\bf Return} $(X'_\bullet,Y'_\bullet)$ \\
			\hline
		\end{tabular}
	\end{center}
\medskip

 From Theorem \ref{thm:Stat_wrt_Flag} and Proposition \ref{propn:JacobianRankRestrictionMap}, we obtain Proposition \ref{propn:refineWorks} below, which is analogous to \cite[Proposition 6.7]{hn}. The novelty here lies in the fact that the result below applies to arbitrary polynomial maps between affine complex varieties, and does not impose any genericity requirements. The proof, however, is essentially identical to that of \cite[Proposition 6.7]{hn}.
	\begin{proposition}\label{propn:refineWorks}
		The {\bf Refine} subroutine terminates, and its output $(X_\bullet,Y_\bullet)$ constitutes a valid stratification (as in Definition \ref{def:stratmap3}) of the polynomial map $f:X \to Y$.
	\end{proposition}\begin{proof}
	    In \cite[Proposition 6.7]{hn} this result is obtained for generic mappings $f:X\to Y$, as defined in \cite[Definition 6.3]{hn}, using the construction of the {\bf Refine} subroutine along with a result (\cite[Proposition 6.6]{hn}) analogous to that of Proposition \ref{propn:JacobianRankRestrictionMap} in the case of generic maps.  The new {\bf Refine} subroutine above is identical to that in \cite[Section 6]{hn} other than the new version uses {\bf WTBStrat} at Line 11, rather than {\bf WhitStrat}. The argument is then identical to that of \cite[Proposition 6.7]{hn}, except we employ the more general result of Proposition \ref{propn:JacobianRankRestrictionMap} in place of \cite[Proposition 6.6]{hn}. 
	\end{proof}
 
\section{Whitney Stratifications from Polar Varieties}\label{sec:polarStrat}
Consider a complex affine variety $X\subset \CC^n$. In this section we present a new algebraic condition to identify (a superset of) those points in a subvariety $Y\subset X_{\rm sing}$ where Whitney's Condition (B) fails with respect to $X$. This criterion replaces Theorem \ref{thm:WhitB_Primary_Decomp} and leads to a new algorithm for computing Whitney stratifications. This new algorithm is based on {\em polar varieties} \cite{FTpolar, Piene2015}.

The polar approach has two key advantages: first, one can perform all computations in the ring $\CC[x_1, \dots, x_n]$ rather than working with the additional $n$ variables $\xi_1, \dots, \xi_n$. And second, one is no longer required to  compute any primary decompositions. Both aspects confer significant practical benefits, as is illustrated in practical tests in Section \ref{sec:performance}; the later point would also be expected to have significant implications for a worst case complexity analysis of the algorithm (though this is not carried out here). 
On the other hand, the new polar algorithm has the disadvantage of being probabilistic\footnote{In the sense that its correctness depends on a choice of random constants being truly random, as will be explained later.}.

\subsection{Polar Varieties and Condition (B)}
Let $X\subset \CC^n$ be a pure $d$-dimensional complex algebraic variety; we recall that the (Zariski open) subset of smooth points is denoted $X_\text{reg}$, and at each smooth point $p \in X_\text{reg}$ there is a well-defined $d$-dimensional tangent space $T_pX_\text{reg}$, which we treat as a linear subspace of $\CC^n$. Consider a flag $L_\bullet$ of length $d$
\[
L_\bullet = \left( L_{1} \supset \cdots \supset L_{d-1} \supset L_{d}\right),
\]
where each $L_i \subset \CC^n$ is an $(n-i)$-dimensional linear subspace.

\begin{definition}\label{def:polar}
  For each $i$ in $\set{1,\ldots,d}$, the codimension $i$ {\bf polar variety} of $X$ along the flag $L_\bullet$ is defined as the closure 
  \[
  \polar_i(X;L_\bullet) := \overline{\set{p \in X_\text{reg} \mid \dim(T_pX_\text{reg} \cap L_{i}) \geq d-i+1}}.
  \] We define $\polar_0(X;L_\bullet):=X$.
\end{definition}

The fact that each $\polar_i(X;L_\bullet)$ is an algebraic variety will be established in the next subsection, where we describe explicit generating equations. For now, let us note that if $X$ is irreducible then $\polar_i(X,L_\bullet)$ is irreducible for all $i$,  and additionally for a fixed flag of linear spaces $L_\bullet$ we obtain a corresponding {\bf polar flag}
\begin{equation}
\polar_d(X,L_\bullet)\subset \cdots \subset \polar_0(X,L_\bullet)=X, \label{eq:polarFlag}   
\end{equation}
(see e.g.~\cite[Remark 3.14]{FTpolar}).

\begin{remark}
One often finds the adjective {\em local} preceding polar varieties in the relevant literature, as for instance in \cite{FTpolar}. To align this definition with ours, one chooses a flag of generic linear spaces $F_\bullet = (F_0 \subset \cdots \subset F_d)$ with $\dim F_i = i$ through a given point $y \in X$, and then defines $L_i$ as the orthogonal complement of $F_i$ in $\CC^n$. 
\end{remark}

Our interest in polar varieties stems from a classical result of Teissier relating Condition (B) to polar multiplicities of points of $X$. In order to state this result, let us recall that the {\bf Hilbert-Samuel multiplicity} of an irreducible subvariety $V \subset X$ is the (integer) coefficient $m_VX > 0$ of $[X]$ in the Segre class $s(X,V)$ --- see \cite[Sec 4.3]{fulton2013intersection}. These multiplicities may be computed numerically using the algorithm of \cite[Theorem 5.3]{HH19} and even inferred from local point samples as described in \cite{clhsm}. Here is Teissier's polar multiplicity criterion \cite[Theorem 4.13]{FTpolar} for Condition (B). 

\begin{theorem}
  \label{thm:Teissiermultseq}
  Let $Y\subset X$ be a pure dimensional subvariety of a pure dimensional variety $X$. The following properties are equivalent: 
  \begin{enumerate} 
  \item The manifolds $(X_\text{\rm reg}, Y_\text{\rm reg})$ satisfy Whitney's Condition (B).
  \item Every point $y \in Y_\text{\rm reg}$ admits an open neighbourhood $U \subset Y_\text{\rm reg}$ such that for every $z \in U$ and generic flag $L_\bullet$, the sequence of Hilbert-Samuel multiplicities 
  \[
  m_{\bullet}(X,z) := \Big(
  m_zX, m_z\polar_1(X,L_\bullet), \dots, m_z\polar_{d-1}(X,L_\bullet)
  \Big),
  \]  
  is constant. 
  \end{enumerate}
  In particular, if we let $Y' \subset Y_\text{\rm reg}$ be the (necessarily open) set of all points which satisfy the second property, then $(X_{\rm reg},Y')$ satisfies Condition (B).
\end{theorem}

\subsection{Computing (Equations of) Polar Varieties}\label{subsec:polarEqs}

As in the previous section, let $X\subset \CC^n$ be a purely $d$-dimensional complex algebraic variety. Its polar varieties may be defined similarly to the conormal variety from \eqref{eq:conormal}, but -- crucially, from an algorithmic perspective -- one only needs to work in the coordinate ring of $\CC^n$. In particular, we define $$
\mathscr{K}_i=\begin{bmatrix}
c^{(0)}_1& \cdots & c^{(0)}_n\\
\vdots& \ddots & \vdots\\
c^{(i)}_1& \cdots & c^{(i)}_n\\
	\frac{\partial f_1}{\partial x_1} &\cdots& \frac{\partial f_1}{\partial x_n}\\
	\vdots& \ddots & \vdots \\
	\frac{\partial f_m}{\partial x_1} &\cdots& \frac{\partial f_m}{\partial x_n}\\
	\end{bmatrix} \quad \text{ and }
 \quad \mathscr{J}=\begin{bmatrix}
	\frac{\partial f_1}{\partial x_1} &\cdots& \frac{\partial f_1}{\partial x_n}\\
	\vdots& \ddots & \vdots \\
	\frac{\partial f_m}{\partial x_1} &\cdots& \frac{\partial f_m}{\partial x_n}\\
	\end{bmatrix},
$$ where the $c^{(i)}_j$ are general constants for $i \in \set{0, \dots, d-1}$. 
Let $K_i$ be the ideal generated by all $(n-d+i+1)\times (n-d+i+1)$ minors of the matrix $\mathscr{K}_i$, and let $J$ be the ideal defined by the $(n-d)\times (n-d)$ minors of the matrix $\mathscr{J}$. Then \begin{equation}
\polar_{d-i}(X,L_\bullet)=\bV_\CC((K_i+I_X):J^\infty),    \label{eq:polarVariety}
\end{equation}
where $L_i$ is the orthogonal complement of the linear space spanned by the first $i$ rows of $\mathscr{K}_i$. 

\begin{remark} The following statements pertain to practical computation of polar varieties.
\begin{enumerate}
\item The ideal saturation required in \eqref{eq:polarVariety} may be preformed via a single elimination calculation using (a probabilistic variation of) the Rabinowitsch trick. Namely, if $J=\langle g_1, \dots, g_\nu \rangle$ and we working in the ring $\CC[x_1,\dots, x_n,T]$ and consider the ideal $$
\mathcal{I}_i:=K_i+I_X+\langle 1-T \sum_{j=1}^\nu \lambda_j g_j \rangle
$$for general constants $\lambda_i\in \CC$.  Then, we have $$
(K_i+I_X):J^\infty= \CC[x_1,\dots, x_n]\cap \mathcal{I}_i. 
$$ For more details, see \cite{hauenstein2022probabilistic}.

\item In practice, the general constants are replaced by random ones and hence the computation detailed above becomes inherently probabilistic. In our setting this is no great loss, as our polar varieties already depend on (essentially random) choices of linear polynomials.
\item Alternatively, the polar variety $\polar_i(X,L_\bullet)$ may be computed by a non-probabilistic elimination in the following way: 
 \begin{itemize}
     \item Consider any matrix $D$ whose columns span $L_i$.
     \item Let $a = (a_1,\dots,a_i)^T$ be a vector of indeterminates. 
     \item The projection of $\{(x,a)\mid x\in X,\ \mathscr{J}Da = 0,\ a\neq 0 \}$ to $X$ is, by definition\footnote{Note that $Da$ for $a\neq 0$ is an element in $L_i$; now the condition $\exists a\in\CC^i$ such that $\mathscr{J}Da = 0$ is satisfied exactly by those points $x\in X$ for which the intersection $\kappa_X^{-1}(x) \cap L_i$ is nonempty.}, the polar variety $\polar_i(X,L)$. .    
     \item We compute the defining ideal of the projection above using saturation and elimination:
     $$(I_X+(\ideal{\mathscr{J}Da}:\ideal{a}^\infty)) \cap \CC[x].$$    
 \end{itemize}
 \end{enumerate}\label{remark:polarEqsV2}
\end{remark}

\subsection{Polar Varieties and Equi-multiplicity}
Our new algorithm to compute a Whitney stratification will make use of the criterion of Teissier from Theorem \ref{thm:Teissiermultseq}. This criterion in turn is phrased in terms of the constancy of the Hilbert-Samauel multiplicity of points relative to a list of varieties; the property of having constant multiplicity relative to a variety for all points in a Zariski open set is often called {\em equi-multiplicity} \cite{Equimultiplicity}. The following lemma gives a simple sufficient condition for equi-multiplicity; in spirit it follows many related results which have been summarized in Theorems 2.2.2 and Theorem 2.2.5 of the Appendix of \cite{Equimultiplicity}. The main difference is that here we use a multiplicity formula from \cite{HH19}, which lends itself to a particularly elementary proof.

\begin{lemma}\label{lemmq:equimult}
Consider varieties $Y\subset X \subset \CC^n$, with $X$ pure dimensional and $Y$ nonempty, irreducible, and different from  $\polar_i(X,L_\bullet)$ for all $i$. Let $\mu =m_YX$ be the Hilbert-Samuel multiplicity of $Y$ in $X$, and let $r$ be the codimension $\dim X - \dim Y$. Then for all points $z\in Y_{\rm reg}- \polar_r(X, L_\bullet)$, we have $m_zX=\mu$. 
\end{lemma}
\begin{proof}
First note that multiplicity is unchanged under projective closure, hence we may replace the setup in the statement of the Lemma with $Y\subset X \subset \PP^n$. By \cite[Sec 2.1.4]{HH19} we may further suppose that $Y=\bV_\CC(f_1,\ldots, f_s)$ where all the $f_i$ have the same degree. Let $\lambda_{ij}$ be an $r \times s$ matrix of generic complex numbers, and for each $i \in \set{1,\ldots,r}$ define the polynomial $Q_i = \sum_j \lambda_{ij}f_j$. Setting $W=\bV_\CC(Q_1, \dots, Q_r)$, Theorem 5.3 of \cite{HH19} gives 
\begin{align}\label{eq:degformula}\deg(Y) \mu=\deg(X)\deg(W)-\deg\left(\overline{(X\cap W)-Y}\right).
\end{align}  
Fix a point $z \in Y_\text{reg}$, and let  $A_z \subset \PP^n$ be a linear space of codimension $\dim(Y)$ which contains $z$. We now examine how \eqref{eq:degformula} changes when all varieties in sight are intersected with $A_z$. For generic choices of $A_z$, it is known (by \cite[Proposition 8.4]{fulton2013intersection} for instance) that $\deg(V \cap A_z) = \deg(V)$ holds for $V \in \set{W,X,Y}$. We remark that a sufficient condition for $A_z$ to be generic in the sense required here is for $A_z$ to be transverse to $X$ in the punctured neighborhood of $z$.  

Now, by \cite[Proposition 3.3]{helmer2023complex}, for generic $A_z$ we also have 
\[
m_zX = m_{Y \cap A_z}(X \cap A_z).
\] 
Therefore,
\begin{align}\label{eqn:ezxmu}
m_zX - \mu = \frac{\deg \left(\overline{(X \cap W) - Y}\right) -\deg\left(\overline{(X \cap W \cap A_z) - (Y \cap A_z)}\right)}{\deg Y}.
\end{align}
Next, we seek to establish that the numerator on the right side vanishes unless $z$ lies in $\polar_r(X;L_\bullet)$. To this end, let us recall that the numbers $\set{\lambda_{ij}}$ which define $W$ are generic. Therefore we may safely assume that $\dim W = n-r$, and that $\dim(X \cap W) = \dim Y$, with the latter equality holding due to transversality. Therefore -- at least when $A_z$ is generic -- both $X \cap W \cap A_z$ and $Y \cap A_z$ are finite sets of cardinality $\deg (X \cap W)$ and $\deg Y$ respectively. Returning to \eqref{eqn:ezxmu}, we conclude that $m_zX = \mu$ holds unless $A_z$ fails to intersect $X_\text{reg}$ transversely near $z$. But since $A_z$ might well equal $L_{\dim Y}$ in the flag $L_\bullet$, we conclude that $m_zX \neq \mu$ forces $z \in \polar_r(X;L_\bullet)$.
\end{proof}

We now employ the preceding result to construct a new Whitney stratification algorithm. It will be convenient to represent an $X_i$ in the flag $X_\bullet$ as a list of irreducible varieties (or equivalently, as a list of prime ideals) whose union equals $X_i$. In this way we will update intermediate values of $X_i$ by appending new components to this list.  For clarity, we will use ${\rm List}(X_i)$ to denote the list of components whose union equals $X_i$. We will also assume that a random flag $L_\bullet$ (as required for Definition \ref{def:polar}) has been chosen in advance. 

\begin{center}
  \begin{tabular}{|r|l|}
    \hline
    ~ & {\bf WhitneyPolar}$(X)$ \\
    \hline
    ~&{\bf Input:} An algebraic variety $X$.\\
    ~&{\bf Output:} $X_\bullet$, the flag in \eqref{eq:WhitneyFlag} corresponding to the Whitney stratification of $X$.\\
    \hline
    1 & {\bf Set} $d:=\dim(X)$; \\
    2 & {\bf Initialize} $X_\bullet$ to be a flag of length $d+1$ with each ${\rm List}(X_{i}) := \varnothing$\\
    3 & {\bf For} each equidimensional component $Z \subset X$, append $Z$ to ${\rm List}(X_{\dim (Z)})$\\
     
    4 & {\bf For} $i$ from $0$ to $d-1$ {\bf do}\\
    5 & \spc {\bf For} $j$ from $1$ to $d-i$ append $\polar_j(X_{d-i},L_\bullet)$ to ${\rm List}(X_{d-i-j})$\\
    6 & \spc {\bf Compute} $Y:=(\bigcup_{j=0}^{i}X_{d-j})_{\rm sing}$\\
    7 & \spc {\bf For} each equidimensional component of $Z \subset Y$, append $Z$ to ${\rm List}\big(X_{\dim (Z)}\big)$\\
    8 & \spc {\bf For} $j$ from $0$ to $i-1$ \\
    9 & \spc \spc $V := X_{d-i} \cap \polar_{i-j}(X_{d-j},L_{\bullet})$\\
    10& \spc \spc {\bf For} each equidimensional component $W \subset V$, append $W$ to ${\rm List}\big(X_{\dim (W)}\big)$ \\
    11& {\bf Return} $X_\bullet$\\
    \hline
  \end{tabular}
\end{center}\medskip
\begin{remark}
    A recent discussion of algorithms and software to perform equidimensional decomposition (along with a new algorithm for this task) can be found, for example, in \cite{eder2023direttissimo}. Note that alternatively, depending on the implementation context, one may choose to take irreducible components above everywhere (i.e.~in Lines 3, 7, and 10) rather than equidimensional ones.  
\end{remark}

\begin{remark}
The {\bf WhitneyPolar} algorithm presents the flag $X_\bullet$ as a union of parts, i.e., a list of components that have the same dimension but are not necessarily irreducible. One may wish to pursue two further goals: 
\begin{enumerate}
    \item design an algorithm to compute the actual connected components in the Whitney stratification, or
    \item produce a refinement of the Whitney stratification which ensures that each component is connected.
\end{enumerate} 
These tasks depend in practice on whether one works over the real or complex numbers. The latter can be achieved (over $\CC$) by decomposing every component in the output into irreducible components. 
\end{remark}

\begin{corollary}
    The algorithm {\bf WhitneyPolar} is correct. 
\end{corollary}
\begin{proof}
    The first three lines of the algorithm merely populate the flag $X_\bullet$ with the equidimensional components of $X$. To prove the correctness of the algorithm, it suffices to show that during the $i$-th iteration of the {\bf For} loop on Line 4, each $(d-i)$-dimensional part of the stratification, $\widetilde X \in {\rm List}(X_{d-i})$ has been correctly processed. Specifically, we show that at step $i$ the algorithm generates a (finite) list of lower-dimensional subvarieties $Y \subset \widetilde X$ such that the union of the generated varieties $Y$ contains all points $y$ which are either not smooth in $X$, or for which the following criterion holds: for every $j > d-i$ and equidimensional component $Z \in {\rm List}(X_j)$ that contains $\widetilde X$, the multiplicity sequence $m_{\bullet}({Z},y)$ differs from the expected sequence $m_{\bullet}({Z},x)$ for  generic $x \in \widetilde X$.
    
    Line 5 records all points on all polar subvarieties of $\widetilde X$; in regards to \Cref{lemmq:equimult}, this deals separately with the case where $Y$ in the Lemma is a polar subvariety (by ensuring all of these are already removed). Lines 6-7 handle singular points of $\widetilde X$. Finally, 
     Lines 8-10 excise the points that are not singular {\em and} do not lie on a polar subvariety, but have to be removed to guarantee the conclusion of \Cref{lemmq:equimult} (note that we identify a variety containing all points where \Cref{lemmq:equimult} can fail for some component of our pure dimensional variety $X_j$ all at once, rather than one irreducible component at a time).  
\end{proof}

\begin{remark} \label{rem:notminimal} The reader is warned that the stratifications produced by {\bf WhitneyPolar} may not be minimal. There are two reasons for this lack of optimality --- first, the condition (from Lemma \ref{lemmq:equimult}) which we use to ensure equi-multiplicity is sufficient but not necessary condition. And second, since our polar varieties are defined globally (rather than locally at a point), they inevitably include various random points arising from the choice of $L_\bullet$ which are not in fact necessary. We are aware of two antidotes to this --- one may either generate a suboptimal stratification and then coarsen it using the algorithm of \cite[Sec 3.3]{helmer2024new}, or alternately, one may replace each polar variety in {\bf WhitneyPolar} with the intersections of several global polar varieties (of the same codimension), each taken along different generic flags $L_\bullet$.
\end{remark}

\section{Real Stratifications }\label{sec:realvarietystrat}

Let $\RR[x_1,\ldots,x_n]$ be the ring of real polynomials in $n$ indeterminates, and fix a radical ideal $I$ of this ring. By definition, the vanishing locus $X := \bV(I)$ constitutes a real algebraic subvariety of $\RR^n$. Let $\set{f_1,f_2,\ldots,f_r}$ be any finite set of polynomials which generate $I$. Since each $f_i$ is automatically a polynomial in $\CC[x_1,\ldots,x_n]$, there is a complex variety $\bV_\CC(f_1,\ldots,f_r) \subset \CC^n$ generated by the $f_i$. Although this variety depends on $\set{f_i}$ rather than on $X$, we will denote it by $X(\CC)$ in order to emphasise that it is a complexified version of $X$. Let $X_{\rm reg}$ denote the manifold of smooth points in $X$. In this section $\dim_{\RR}(X)$ will denote the dimension of manifold $X_{\rm reg}$ and $\dim_{\CC}(X(\CC))$ will denote the dimension of the manifold $(X(\CC))_{\rm reg}$.

\subsection{Real Varieties} 
Our immediate goal here is to show that certain Whitney stratifications of $X(\CC)$ induce Whitney stratifications of $X$. We let $\iota:\RR^n \inj \CC^n$ be the embedding of real points in complex Euclidean space.

\begin{lemma}\label{lemma:SmoothCtoR}
Let $X \subset \RR^n$ be a real algebraic variety.
\begin{enumerate}
    \item The embedding $\iota$ identifies $X$ with the real points of $X(\CC)$.
    \item Assume that $\dim_\RR X$ equals $\dim_\CC X(\CC)$. If $\iota(p)$ is a smooth point of $X(\CC)$ for some $p \in X$, then $p$ is a smooth point of $X$.
\end{enumerate}
\end{lemma}
\begin{proof}
The first assertion is a tautology. Turning to the second assertion, set $d := \dim_\CC X(\CC) = \dim_\RR X$. Since the roots of real polynomials occur in complex conjugate pairs, the variety $X(\CC)$ is invariant under complex conjugation and $\iota(X)$ equals the fixed point set of this conjugation. Noting that $\iota(p)$ is a smooth point by assumption, the tangent space $T_{\iota(p)}X(\CC)$ exists and has complex dimension $d$; this tangent space also inherits invariance under complex conjugation. Thus, $T_{\iota(p)}X(\CC)$ is the complexification of a real $d$-dimensional vector space $V$ whose elements consist of all real tangent vectors at $\iota(p)$; this $V$ is evidently isomorphic to $T_pX$, as desired.
\end{proof}

Every finite descending chain $I_\bullet$ of radical ideals in $\RR[x_1,\ldots,x_n]$
\[
I_0 \rhd I_1 \rhd \cdots \rhd I_m = I
\]
produces an ascending flag $X_\bullet := \bV(I_\bullet)$ of subvarieties of $X$:
\[
X_0 \subset X_1 \subset \cdots \subset X_m = X.
\]
The next result shows that successive differences of $X_\bullet$ inherit a smooth manifold structure from the successive differences of $X_\bullet(\CC)$.

\begin{proposition}\label{prop:CtoRmfd} Let $W \subset Z$ be a pair of real algebraic varieties in $\RR^n$. If the difference $Z(\CC)-W(\CC)$ is either empty or a smooth $i$-dimensional complex manifold, then $M := (Z-W)$ is either empty or a smooth $i$-dimensional real manifold. 
\end{proposition}
\begin{proof}
 There are two cases to consider --- either the image $\iota(Z)$ lies entirely within the singular locus $Z(\CC)_{\rm sing}$, or there exists some $p \in M$ with $\iota(p) \in Z(\CC)_{\rm reg}$. In the first case, since $Z(\CC)-W(\CC)$ is smooth, we know that $Z(\CC)_{\rm sing}$ lies entirely within $W(\CC)$ and hence that $\iota(Z) \subset \iota(W)$; but since we have assumed $W \subset Z$, we must have $W=Z$, whence $M$ is empty. On the other hand, let $p$ be a point in $(Z-W)$ for which $\iota(p)$ is a smooth point of $Z(\CC)$. We may safely assume that the generating ideal of $Z$ is prime in $\CC[x_1,\ldots,x_n]$ by passing to the irreducible component which contains $\iota(p)$. It now follows from \cite[Theorem~12.6.1]{marshall2008positive} or \cite[Theorem~2.3]{harris2023smooth} that $(Z-W)$ has dimension $i$. Finally, Lemma \ref{lemma:SmoothCtoR} ensures that $(Z-W)$ is a smooth real $i$-manifold.   
\end{proof}

It follows from the above result that if $X_\bullet(\CC)$ is a Whitney stratification of $X(\CC)$, then the successive differences of $X_\bullet$ are either empty or smooth manifolds of the expected dimension. We show below these successive differences also satisfy Condition (B).
 
\begin{theorem}
Let $I_\bullet$ be a descending chain of radical ideals in $\RR[x_1,\ldots,x_n]$. If the flag $X_\bullet(\CC) := \bV_\CC(I_\bullet)$ constitutes a Whitney stratification of $X(\CC)$, then the corresponding flag $X_\bullet := \bV(I_\bullet)$ yields a Whitney stratification of $X$. \label{thm:RealWhitStrat}
\end{theorem}
\begin{proof}
Consider a connected component $M \subset (X_i - X_{i-1})$, and let $V \subset X_i$ be the irreducible component which contains $M$. Similarly, let $W \subset X_i(\CC)$ be the irreducible component which contains $\iota(M)$ and define $M_{\CC} := W - X_{i-1}(\CC)$. We note that $\iota(M)$ forms an open subset of $M_\CC$, which must in turn be an $i$-stratum of $X(\CC)$. Similarly, consider a nonempty connected $N \subset (X_j - X_{j-1})$ with $i > j$ and analogously define $N_\CC \subset X_j(\CC)$. We will show that the pair $(M,N)$ satisfies Condition (B). To this end, consider a point $q \in N$ along with sequences $\set{p_k} \subset M$ and $\set{q_k} \subset N$ which converge to $q$. Letting  $\ell_k$ denote the secant line $[p_k,q_k]$ and $T_k$ the tangent plane $T_{p_k}M$, we assume further that the limits $\ell = \lim \ell_k$ and $T = \lim T_k$ both exist. 

Let $\ell_k(\CC)$ be the secant line $[\iota(p_k),\iota(q_k)]$ in $\CC^n$ and let $T_k(\CC)$ be the tangent space $T_{\iota(p_k)}M_\CC$. Since $\iota(p_k)$ and $\iota(q_k)$ are real points for all $k$, the the linear equations defining both $\ell_k(\CC)$  and $T_k(\CC)$ as varieties are exactly the same as those defining $\ell_k$ and $T_k$, respectively. Thus, the limits $\ell(\CC)$ and $T(\CC)$ both exist because the corresponding real limits exist -- one may view these as limits of real sequences inside a complex Grassmannian -- and they are defined by the same algebraic equations as their counterparts $\ell$ and $T$. By definition of secant lines, the image $\iota(w)$ of any $w \in \ell$ is a real point of $\ell(\CC)$. Since the pair $(M_\CC,N_\CC)$ satisfies Condition (B) by assumption, we know that $\ell(\CC) \subset T(\CC)$, whence $\iota(w)$ must be a real point of $T(\CC)$. Since $\iota(M)$ is an open subset of $M_\CC$, we have $T_{\iota(p_k)}M_\CC = T_{\iota(p_k)}\iota(M)$ for all $k$; and by the proof of Lemma \ref{lemma:SmoothCtoR}, the real points of $T_{\iota(p_k)}\iota(M)$ are identified with $T_{q_k}N$. Thus, $\iota(T)$ contains all the real points of $T(\CC)$, including $\iota(w)$. Since $\iota$ is injective, we have $w \in T$ as desired. \end{proof}

\begin{corollary}\label{cor:whitstratworks}
Let $X$ be a real algebraic variety. If either the {\bf WhitStrat} algorithm from Section \ref{subsec:stratificationReview} or the {\bf WhitneyPolar} algorithm from Section \ref{sec:polarStrat} are applied to $X(\CC)$, then the output produces a Whitney stratification of $X$. 
\end{corollary}

\begin{proof}  From an algebraic point of view our input to both the {\bf WhitStrat} and the {\bf WhitneyPolar} algorithm is a list of polynomials in the ring $\RR[x_1, \dots, x_n]$\footnote{Note, in a  realistic computational setting the input is likely to actually be a list of polynomials in $\QQ[x_1, \dots, x_n]$, and all subsequent algebraic operations will occur in this ring.} which define the  real variety $X\subset \RR^n$ and the corresponding complex variety $X(\CC)\subset \CC^n$; in the algorithms we then perform algebraic operations starting with the polynomials in this list. We need only show that the concrete algebraic operations performed in both the {\bf WhitStrat} algorithm and the {\bf WhitneyPolar} algorithm leave the coefficient field of all intermediate polynomials unchanged. To see this we list the algebraic operations performed, the conclusion then follows from an analysis of the underlying algorithms in standard references such as \cite{CLO}.

First we list the geometric operations which occur in the {\bf WhitStrat} algorithm along with their algebraic counterparts in the usual algebraic-geometric dictionary, see e.g.~\cite{CLO}. \begin{enumerate}
    \item Intersection of varieties, which corresponds to addition of polynomial ideals. 
    \item Computation of the singular locus of a variety, from Section \ref{subsec:stratificationReview} (on page 5) this corresponds to computing the radical of the polynomial ideal given by taking the appropriately sized minors of the Jacobian matrix of the defining equations. 
    \item Computing the conormal variety, from Section \ref{subsec:stratificationReview} (on page 5) this corresponds the computation of minors of a matrix of polynomials (most of which arise from partial derivatives), and the addition and saturation of polynomial ideals. 
    \item Computing the irreducible components of a variety, which corresponds to prime decomposition of ideals. 
     \item Computing the primary decomposition of ideals. 
\end{enumerate}
Addition of ideals is in effect appending polynomials to a list, which leaves the coefficient field unchanged. Polynomial differentiation and the taking of minors of polynomial matrices clearly leave the coefficient field unchanged as well. Ideal saturation as well as the computation of the radical and the prime and primary decomposition can be accomplished via Gr\"obner basis computation, see \cite{CLO} for the saturation computation and \cite{eisenbud1992direct} for the radical, prime and primary decomposition computations.  From a standard presentation of Buchberger's algorithm to compute Gr\"obner basis, e.g.~\cite[Chapter 2, \S8]{CLO}, we see this leaves the coefficient field unchanged. Hence all algebraic operations in items (1)--(5) leave the coefficient field unchanged and the conclusion follows. 

We note that the list of algebraic options above includes computing the radical of an ideal, so for the purposes of this argument it does not matter if the input list of polynomials defines a radical ideal or not. 

Similarly, the {\bf WhitneyPolar} algorithm performs items (1), (2). The additional operations are: 
\begin{enumerate}\setcounter{enumi}{5}
\item Equidimensional decomposition of a variety (or we could use irreducible decomposition in the relevant lines, in which case this is covered by (5)). 
\item The computation of polar varieties; by equation \eqref{eq:polarVariety} in Section \ref{subsec:polarEqs} this corresponds the computation of minors of a matrix of polynomials (most of which arise from partial derivatives), and the addition and saturation of polynomial ideals. 
\end{enumerate}
It can be seen from \cite{eisenbud1992direct} that equidimensional decomposition leaves the coefficient field unchanged and, as discussed above, this is also true of the polynomial differentiation, the computation of minors of matrices of polynomials, and saturation of polynomial ideals. The conclusion follows. 
\end{proof}

\subsection{Real Algebraic Morphisms}\label{section:realVarietyMapStrat}

Consider algebraic varieties $X \subset \RR^n$ and $Y \subset \RR^m$, and let $f:X\to Y$ be an algebraic morphism. We recall that this amounts to an $m$-tuple of real polynomials 
\[
\Big(f_1(x_1,\ldots,x_n), ~ f_2(x_1,\ldots,x_n), ~ \ldots, ~ f_m(x_1,\ldots,x_n)\Big)
\] 
whose evaluation at a point of $X$ yields a point of $Y$. Since each $f_i$ is automatically a complex polynomial, there is an evident morphism $f_\CC:X(\CC) \to Y(\CC)$ of complex algebraic varieties. Let $I_\bullet$ and $J_\bullet$ be descending chains of radical ideals in $\RR[x_1,\ldots,x_n]$ and $\RR[y_1,\ldots,y_m]$ respectively so that $X_\bullet(\CC) := \bV_\CC(I_\bullet)$ and $Y_\bullet(\CC) := \bV_\CC(J_\bullet)$ constitute Whitney stratifications of $X(\CC)$ and $Y(\CC)$ respectively. It follows from Theorem \ref{thm:RealWhitStrat} that $X_\bullet := \bV(I_\bullet)$ is a Whitney stratification of $X$ while $Y_\bullet := \bV(J_\bullet)$ is a Whitney stratification of $Y$. 

 \begin{theorem} \label{thm:realmapstrat}
If $f_\CC$ is stratified with respect to $X_\bullet(\CC)$ and $Y_\bullet(\CC)$, then $f$ is stratified with respect to $X_\bullet$ and $Y_\bullet$. \label{thm:realMapsStrat}
\end{theorem}
\begin{proof} 
Let $M \subset X$ be a nonempty connected component of the $i$-stratum $X_i-X_{i-1}$, and let $M_\CC$ be the $i$-stratum of $X_\bullet(\CC)$ which contains $\iota(M)$. By definition, the image $f_\CC(M_\CC)$ contains $f(M)$ in its locus of real points. Since $f_\CC$ is stratified with respect to $X_\bullet(\CC)$ and $Y_\bullet(\CC)$, the first requirement of Definition \ref{def:stratmap3} guarantees the existence of a single stratum $N_\CC \subset Y$ which contains $f_\CC(M_\CC)$. Thus, $f(M)$ lies in the locus of real points of $N_\CC$. Letting $N$ denote the stratum of $Y_\bullet$ corresponding to $N_\CC$, we know that the real locus of $N_\CC$ equals $N$, whence we obtain $f(M) \subset N$ and it remains to show that the restriction of $f$ to $M$ yields a submersion. Let $x$ be any point of $M$, and note that $\iota(x)$ lies in $M_\CC$. Since $f_\CC|_{M_{\CC}}$ is a submersion, its Jacobian at $\iota(x)$ is a surjective linear map from the tangent space to $M_\CC$ at $\iota(x)$ to the tangent space to $N_\CC$ at $f_\CC\circ\iota(x)$. But by construction, $f_\CC \circ \iota$ equals $f$. Thus, we have
\[
\text{rank}_\CC \left(\jac f_\CC|_{M_\CC}(\iota(x))\right) = \dim_\CC N_\CC.
\]
To conclude the argument, we note that the derivative arising on the left side of the above equality may be represented by the Jacobian matrix of $f$ at $x$, and the rank of this matrix is preserved under field extension to $\CC$. On the other hand, by Proposition \ref{prop:CtoRmfd} we know that the complex dimension of $N_\CC$ equals the real dimension of $N$. Thus, our equality becomes
\[
\text{rank}_\RR\left(\jac f|_M(x)\right) = \dim_\RR N,
\]
as desired.
\end{proof}

Assume that a morphism $f:X \to Y$ has been stratified as described in Theorem \ref{thm:realMapsStrat}. It is readily checked that the image $f(\overline{M})$ of the closure of a stratum $M \subset X$ is not an algebraic subvariety of $Y$ in general --- the best that one can expect is that $f(\overline{M})$ will be semialgebraic. It is important to note, in the context of the above theorem, that we do not obtain semialgebraic descriptions of such images.
\label{remark:semiAlgImages}

\subsection{Dominant Morphisms between Equidimensional Varieties} \label{section:DominantMapStrat}

Let $X \subset \mathbb{K}^n$ and $Y \subset \mathbb{K}^m$ be algebraic varieties defined by ideals $I_X$ and $I_Y$ over a field $\mathbb{K} \in \set{\RR,\CC}$. Let $\mathbb{K}[X]$ denote the coordinate ring $\mathbb{K}[x_1,\ldots,x_n]/I_X$ and similarly for $Y$; any morphism of varieties $f:X \to Y$ canonically induces a contravariant ring homomorphism $f^*:\mathbb{K}[Y] \to \mathbb{K}[X]$. 

\begin{definition}\label{def:dominant}
Let $f:X \to Y$ be a morphism of algebraic varieties over $\mathbb{K} \in \set{\RR,\CC}$. 
\begin{enumerate}
    \item we say that $f$ is {\bf dominant} if $f^*$ is a monomorphism (or equivalently, if the image $f(X)$ is dense in $Y$).
    \item we say that $f$ is {\bf finite} if it is dominant, and moreover, if $f^*$ gives $\mathbb{K}[X]$ the structure of an integral extension of $\mathbb{K}[Y]$.
\end{enumerate}
\end{definition}
\noindent It is a classical result, see e.g.~\cite[Chapter II, Corollary 6.7]{Mumford2015}, 
that if $f:X \to Y$ is finite in the above sense, then it is also a proper map for any field (see Remark \ref{remark:fibration} for the significance of this result in our context). Over $\mathbb{K}=\CC$, a dominant morphism is finite if and only if it is proper.

Throughout this section, $f:X\to Y$ will denote a morphism between real algebraic varieties of the same dimension $d$; we will further require that the map $f_\CC:X(\CC) \to Y(\CC)$ is dominant and that $\dim_\CC X(\CC) = \dim_\CC Y(\CC) = d$. 

\begin{definition}\label{def:jelonekset}
     The {\bf Jelonek set} of $f$ is the subset $\Jel(f) \subset Y(\CC)$ consisting of all points $y$ for which there exists a sequence $\set{x_k} \subset X(\CC)$ satisfying both 
     \[\lim_{k \to \infty} |x_k| = \infty \quad \text{ and } \quad \lim_{k\to \infty} f_\CC(x_k) = y.
     \]
 \end{definition}

\noindent It is shown in \cite{jelonek1999testing} that $\Jel(f)$ is either empty or an algebraic hypersurface of $Y(\CC)$; a Gr\"obner basis algorithm for computing the Jelonek set is given in \cite{stasica2002effective}. It follows from this algorithm that if the Jelonek set is non-empty, then it is defined by a polynomial with real coefficients. Finally, it is shown in \cite{jelonek1999testing} that $\Jel(f)$ is precisely the locus of points at which $f$ fails to be finite. Thus, if we define 
\[
V(\CC) := \Jel(f) \quad \text{and} \quad W(\CC) := \overline{f^{-1}(V)},
\]
then the restriction of $f$ forms a proper map $(X(\CC)-W(\CC)) \to (Y(\CC)-V(\CC))$ --- see \cite[Proposition~6.1]{jelonek2002geometry} for details. Note that the polynomials defining $V(\CC)$ and $W(\CC)$ are real; take $V$ to be the real zero set of the polynomials defining $V(\CC)$, and similarly let $W$ be the real zero set of the polynomials defining $W(\CC)$. It follows immediately that the restriction of $f$ to the difference $(X-W)$ constitutes a proper map to the difference $(Y-V)$.

   \begin{definition}
   The {\bf Jelonek flag} of $f:X\to Y$ is a pair $(W_\bullet,V_\bullet)$ 
   of flags, both of length $d = \dim X = \dim Y$:
   \begin{align*}
        &\emptyset =W_{-1}\subset W_0\subset \cdots \subset W_d=X \\
        & \emptyset =V_{-1}\subset V_0\subset \cdots \subset V_d=Y ,  
   \end{align*}
   defined via reverse-induction on $i \in \set{d-1,d-2,\ldots,1,0}$ as follows. Starting with $V_{d-1}$ as the real points of $\Jel(f)$, we
  \begin{enumerate}
    \item let $W_i$ be $\overline{f^{-1}(V_i)}$, and
    \item let $V_{i-1}$ be the real points in $\Jel(f|_{W_i}:W_i \to V_i)$.
  \end{enumerate}
   
   \end{definition}

\noindent By construction of the Jelonek flag, at each $i$ we have the following alternative:
\begin{enumerate}
    \item {\bf if $V_i(\CC)$ is nonempty}, then $\dim V_i(\CC) = W_i(\CC) = i$ and the restriction of $f$ forms a proper map ${(V_i-V_{i-1}) \to (W_i-W_{i-1})}$; otherwise,
    \item {\bf if $V_i(\CC)$ is empty}, then $f|_{W_i}:W_i \to V_i$ is not dominant; on the other hand, $f|_{W_{i-1}}$ is dominant, but with $\dim W_{i-1}(\CC) > \dim V_{i-1}(\CC)$.
\end{enumerate}

 \begin{theorem}\label{thm:dominantstrat} Let $(W_\bullet,V_\bullet)$ be the Jelonek flag of $f:X \to Y$.  Assume that $X_\bullet$ is a Whitney stratification of $X$ subordinate $W_\bullet$ and that $Y_\bullet$ is a Whitney stratification of $Y$ subordinate to $V_\bullet$ such that $f$ is stratified with respect to $X_\bullet$ and $Y_\bullet$ (in the sense of Definition \ref{def:stratmap3}). Whenever $\dim V_i = \dim W_i$ holds, we have that for every $Y_\bullet$-stratum $R \subset Y$ with $R\subset V_i-V_{i-1}$, the map $f|_{f^{-1}(R)}:f^{-1}(R)\to R$ is a locally trivial fiber bundle.\label{thm:properStrat}
 \end{theorem}\begin{proof}
     Since the stratification $Y_\bullet$ of $Y$ is subordinate to  $V_\bullet$, for each stratum $R \subset Y$ we have that $R\subset V_i-V_{i-1}$ for some $i$, and also since $X_\bullet$ is subordinate to  $W_\bullet$ we also have $f^{-1}(R)\subset W_i-W_{i-1}$. If  $\dim(V_i)=\dim(W_i)$ holds, then the map $f:(W_i-W_{i-1})\to (V_i-V_{i-1})$ is proper. An appeal to the semialgebraic version of Thom's first isotopy lemma \cite[Theorem 1]{coste1995thom} achieves the desired result.
 \end{proof}

\section{Full Semialgebraic Stratifications}\label{sec:semialgstrat}

A (basic, closed) {\em semialgebraic set} is any subset $B \subset \RR^n$ which can be expressed as an intersection of the form $B := X \cap C$ where $X$ is a real algebraic subvariety of $\RR^n$, while the set $C$, called an {\em inequality locus}, is given as follows:
\begin{align}\label{eq:ineqloc}
C := \set{x \in \RR^n \mid g_i(x) \geq 0 \text{ for }0 \leq i \leq k}.
\end{align}
Here the $g_i$'s are a finite collection of polynomials in $\RR[x_1,\ldots,x_n]$. By convention, when the number of inequalities $k$ equals zero, we have $C = \RR^n$. Thus, every algebraic variety is automatically a semialgebraic set in the above sense. The sets $X$ and $C$ are not uniquely determined for a given $B$ in general --- we may, for instance, safely remove any polynomial generator $f:\RR^n \to \RR$ from the defining ideal of $X$ while adding $f \geq 0$ and $-f \geq 0$ to the inequality locus. It is therefore customary to omit $X$ entirely and simply define $B$ as the set of points which satisfy a collection of polynomial inequalities. We find it convenient to write $B = X \cap C$ here because this allows us to isolate a relevant sub-class of semialgebraic sets.

\begin{definition}\label{def:full}
A semialgebraic set $B \subset \RR^n$ is called {\bf full} if it admits an inequality locus $C$ of the form \eqref{eq:ineqloc}, with the additional requirement that its subset
\[
C^\circ := \set{x \in \RR^n \mid g_i(x) > 0 \text{ for }0 \leq i \leq k}
\] is an $n$-dimensional smooth manifold whose closure equals $C$. 
\end{definition}

The following example highlights an important class of full semialgebraic sets which arise in the Morse theoretic study of complex analytic varieties --- see \cite[II.2]{SMTbook}.

\begin{example}
Whitney showed in \cite{whitney1965tangents} that every complex analytic variety $W \subset \CC^n$ admits a Whitney stratification $W_\bullet$ where each intermediate $W_i$ is a complex analytic subvariety. Associated to each $i$-stratum $M \subset W_i-W_{i-1}$ is a stratified space $\mathscr{L} = \mathscr{L}_M$, which is called the {\bf complex link} of $M$ and defined as follows. Consider a point $p \in M$ and let $A \subset \CC^n$ be a generic $(n-i)$-dimensional affine subspace which intersects $W$ transversely at a point $p \in M$. Let $B_\epsilon(p)$ denote the disk containing all points within a sufficiently small radius $\epsilon > 0$ around $p$ with respect to the standard euclidean distance in $\CC^n$. Finally, let $\pi:A \to \CC$ be a generic linear form sending $p$ to zero. For some positive $\delta \ll \epsilon$, one defines
\[
\mathscr{L}_M := W \cap B_\epsilon(p) \cap \set{\pi=\delta}.
\]
The stratified homeomorphism type of $\mathscr{L}_M$ depends only on the stratum $M$, and not on the auxiliary choices of $p, A, \epsilon, \pi$ and $\delta$. Complex links provide {\em normal Morse data} for critical points of stratified Morse functions $f:W \to \RR$. Although such a $\mathscr{L}_M$ will not be a complex (or even real) variety in general, it is always a full semialgebraic subset of the affine space $\set{\pi=\delta}$. In particular, its inequality locus is the disk $B_\epsilon(p) \cap \set{\pi=\delta}$.
\end{example}

Given such an inequality locus $C$ of a full semialgebraic set, we call $C^\circ$ its interior and define its boundary as the difference
\[
\partial C := C - C^\circ.
\] 
This boundary is a semialgebraic subset of the real algebraic variety \begin{equation}
    Y_C := \bV\left(\prod_1^k g_i\right).\label{eq:YCDef}
\end{equation} We adopt the usual convention that the product over the empty set equals $1$, which forces $\partial C = Y_C = \varnothing$ when $k=0$.

\begin{example}\label{ex:toysemi}
Below is a toy example of a full semialgebraic set $B = X \cap C$ inside $\RR^2$. Here $X$ is the union of red parabolas $P_1$ and $P_2$ while $C$ is the region bounded by the black lines $L_1, L_2, L_3$ and the black parabola $P_3$:

\begin{center}
\includegraphics[scale=.5]{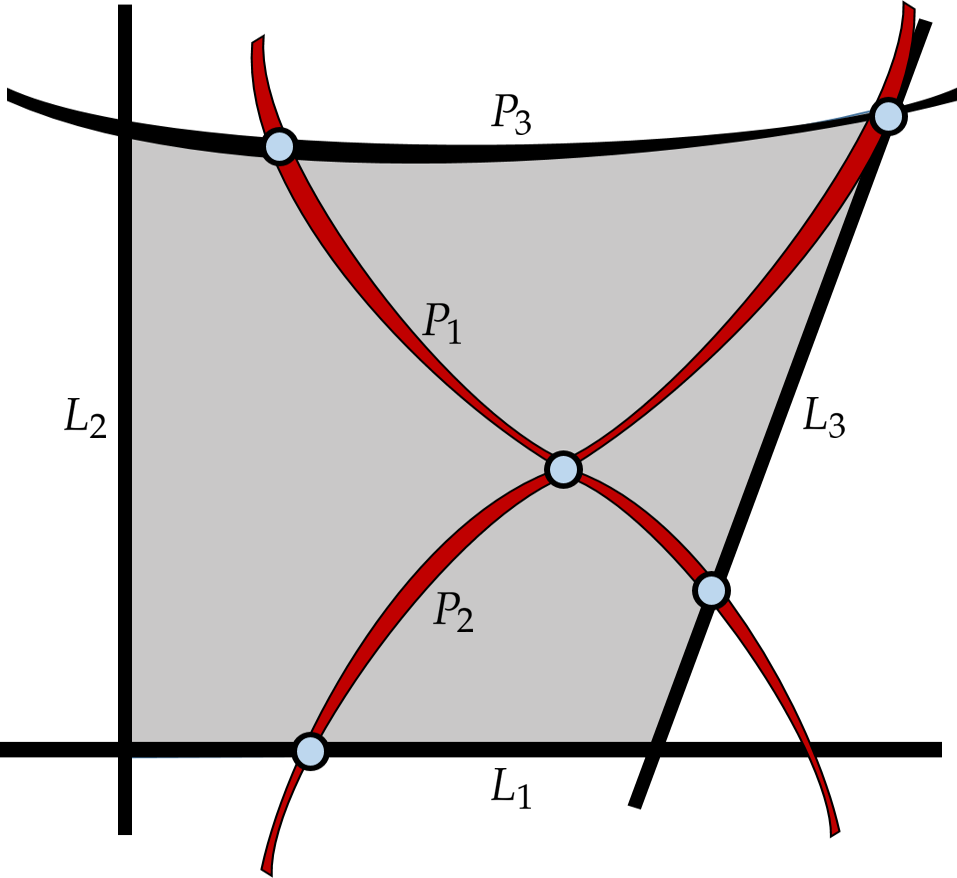}
\end{center}

\noindent The interior $C^\circ$ has been shaded gray, so $B$ consists of that part of $P_1 \cup P_2$ which lies within the closure of the gray region. The variety $Y_C$ is the union of curves $L_1,L_2,L_3$ and $P_3$. The boundary $\partial C$ consists of the four highlighted points  which lie on this union of curves.
\end{example}
 
Our next result establishes that every full semialgebraic set $B = X \cap C$ inherits a Whitney stratification from Whitney stratifications of the real algebraic varieties $X$ and $X \cap Y_C$. (We recall for the reader's convenience that flag-subordinate stratifications have been described in Section \ref{subsec:flagStrat}). 

\begin{theorem} Let $B = X \cap C$ be a full semialgebraic set, with $Y_C$ as in \eqref{eq:YCDef}, and let $X_\bullet$ be a Whitney stratification of $X$. If $Y_\bullet$ is a Whitney stratification of $X \cap Y_C$ which is subordinate to the flag $X_\bullet \cap Y_C$, then setting
\[
B_i : =(X_i\cup Y_i) \cap C
\] produces a Whitney stratification of $B$.
\label{thm:stratSemiAlg}
\end{theorem}

\begin{proof} Since $B$ is full, we know that the interior $C^\circ$ of its inequality locus is a smooth open $n$-dimensional submanifold of $\RR^n$. This means any open neighborhood of $X$, which is contained in $C^\circ$, inside $\RR^n$ is identical to the same neighborhood of $X$ inside $C^\circ$ and therefore, the intersections $X_i \cap C^\circ$ form a Whitney stratification $X'_\bullet$ of $X \cap C^\circ$. Alternatively the fact that $X_i \cap C^\circ$ forms a Whitney stratification $X'_\bullet$ of $X \cap C^\circ$ follows since $C^\circ$ is full dimensional, meaning the intersection $X \cap C^\circ$ is transverse and hence the conclusion follows because the Whitney stratification of two spaces which intersect transversely is given by their intersection,  \cite{cheniot1972sections}, see also \cite[Chapter 1, \S1.2, (5)]{SMTbook}. 

Let $Y'_\bullet$ be the subset of $Y_\bullet$-strata which intersect $\partial C$. Since $C$ is the disjoint union of $C^\circ$ and $\partial C$, it follows that the union of $X'_\bullet$-strata and $Y'_\bullet$-strata  partitions $B$. It remains to check that Condition (B) holds for those strata pairs $(M,N)$ of this union for which $N$ intersects the closure of $M$. There are now three cases to consider, of which the two easy ones are handled as follows:
\begin{enumerate}
\item if both $M$ and $N$ are strata of $X'_\bullet$, then Condition (B) holds because both are full-dimensional open subsets of $X_\bullet$-strata by construction, and $X_\bullet$ is assumed to be a Whitney stratification.
\item if $M$ is a $Y'_\bullet$-stratum, then the fact that $N$ intersects the closure of $M$ forces $N$ to be contained in $X \cap \partial C$, since both $X$ and $\partial C$ are closed subsets of $\RR^n$. Thus, $N$ must also be a $Y'_\bullet$-stratum in which case Condition (B) holds because $Y'_\bullet$ is Whitney.
\end{enumerate}
Turning now to the third case, assume that $M$ is an $X'_\bullet$-stratum and $N$ is a $Y'_\bullet$-stratum. By construction, $M$ must be (a connected component of) the intersection $M_* \cap C^\circ$ for some $X_\bullet$-stratum $M_*$. Since $C^\circ$ is $n$-dimensional, the tangent spaces $T_xM$ and $T_xM_*$ coincide for every $x$ in $M$. Fix a point $p \in N$, and let $N_*$ be the unique $X_\bullet$-stratum containing $p$ in its interior. Since $Y_\bullet$ is chosen subordinate to $X_\bullet\cap Y_C$, the $Y_\bullet$-strata are refinements of $X_\bullet$-strata, so $N$ must be obtained by removing some (possibly empty) set from $N^*\cap \partial C$. It follows that $N$ is a subset of $N^*$ in a small ball around $p$. Finally, $(M,N)$ must satisfy Condition (B) at $p$ because $(M_*,N_*)$ satisfy Condition (B) at $p$. 
\end{proof}

The stratifications obtained in Theorem \ref{thm:stratSemiAlg} provide a complete description of the flag $B_i$. Using the techniques of \cite{harris2023smooth, le2004computing, safey2003polar}, one can perform various fundamental algorithmic tasks involving such strata. These include testing whether the $i$-stratum $B_i - B_{i-1}$ is empty for each $i$, and sampling points from the non-empty strata.

\begin{remark}
Using Theorem \ref{thm:stratSemiAlg} in practice on a full semialgebraic set $B = X \cap C$ requires constructing a stratification of the real variety $X \cap Y_C$; this is somewhat unsatisfactory from an efficiency perspective, since $X$ and $Y_C$ may intersect far away from $B$ --- returning to Example \ref{ex:toysemi}, note that $P_1$ and $P_2$ might intersect $L_2$ at points far from $B$. We are not aware of any existing technique which can easily overcome such suboptimality.
\end{remark}

\section{Performance Analysis and Comparison}\label{sec:performance}

In this section we illustrate the  performance of the {\bf WhitneyPolar} algorithm of Section \ref{sec:polarStrat}. The results have been collected in Table \ref{tab:example}. 

\subsection{Other Methods} 

We have compared {\bf WhitneyPolar} to all known methods for computing practical Whitney stratifications in practice; all of these are based on the conormal stratification procedure reviewed in \Cref{subsec:stratificationReview}. This comparison occupies the first three columns of Table \ref{tab:example}. For reference, we also give the run time for computing a Cylindrical Algebraic Decomposition (CAD) for each given variety. Conversely, we do not include the method of \cite{dhinh2019thom} in our results because it was unable to stratify the Whitney umbrella -- which is the simplest nontrivial input -- after running for 24 hours. We also implemented the algorithm of \cite{rannou1991complexity,rannou1998complexity} for the Whitney umbrella, using the standard \texttt{QuantifierElimination} package in \texttt{Maple}. This code ran for slightly longer than 24 hours, using over 50 GB of RAM (on a workstation with an Intel Xeon W-3365 processor and 1000 GB of RAM). It ultimately crashed with a "Stack Limit Reached" error, even if we allowed unlimited stack size. 

\subsection{Sample Code}

An example of the snippet of code which is run in Macaulay2 \cite{M2} applied to example $X_1$ from Table \ref{tab:example} is given below\footnote{Note that running this code successfully requires Macaulay2 version 1.25.06 or higher and version 2.23 or higher of the \texttt{WhitneyStratifications} \cite{WhitneyStratificationsSource} package.}: 
\begin{lstlisting}
needsPackage"WhitneyStratifications"
R=QQ[x_1..x_4]
I=ideal(x_1^6+x_2^6+x_1^4*x_3*x_4+x_3^3)
elapsedTime V=whitneyStratify(I)
elapsedTime W=whitneyStratify(I,AssocPrimes=>false)
elapsedTime Vpol=whitneyStratifyPol(I,Algorithm=>"msolve")
\end{lstlisting} 
In the code snippet above the stratification \texttt{V} is computed using the method of \cite{hn}, the stratification \texttt{W} is computed using the method of \cite{helmer2024new}, and the stratification \texttt{Vpol} is computed using the {\bf WhitneyPolar} algorithm introduced above in Section \ref{sec:polarStrat}. 

\subsection{Gr\"obner Basis Computation}

For all Whitney stratification algorithms demonstrated in Table \ref{tab:example}, we have provided the run time both over $\mathbb{F}_{32749}$, a finite field of characteristic 32749, and over the rationals. In practice, efficient implementations of Gr\"obner basis algorithms over $\QQ$ often first compute over different large finite fields, and then employ rational reconstruction (this is used in \cite{msolve} for instance). A similar multi-modular philosophy could be applied to Whitney stratification computation as a whole. On larger examples, the timings in the table seem to indicate there could be a significant benefit to doing so. 

\subsection{Hardware}
All computations in the first three columns of Table \ref{tab:example} were run using Version 2.23 of the \texttt{WhitneyStratifications} \cite{WhitneyStratificationsSource} package in Macaulay2 \cite{M2} version  1.25.06. The last column used the \texttt{CylindricalDecomposition} command in \texttt{Mathematica} and was run with Wolfram 14.2.1. All computations were carried out on a laptop with an Intel Ultra 7 processor (258V) and 32 GB of RAM. The {\bf WhitneyPolar} algorithm uses the \texttt{msolve} \cite{msolve} Gr\"obner basis library (version 0.9.0) via the \texttt{Msolve} Macaulay2 package. Computations which used greater than 32GB we also run again on a desktop with 64GB of RAM available; in {\em all} cases these computations again exhausted all available RAM before the end of the 11 hour run. On all runs, {\bf WhitneyPolar} used no more that 400MB of memory.

\subsection{The Input Varieties}
Below we give the equations for the examples presented in Table \ref{tab:example}. The Whitney Cusp and Umbrella are the simplest examples of singular varieties whose Whitney stratification can not be obtained by computing iterated singular loci. 
Examples $X_1$ through $ X_4$ are taken from \cite{helmer2024new}; these are more complicated and also have non-trivial Whitney stratifications. 

\begin{equation}
X_1=\bV\left(x_1^6+x_2^6+x_1^4x_3x_4+x_3^3\right)
\subset \CC^4
\label{eq:X1Ex}
\end{equation}

\begin{equation}
X_2=\bV\left(x_{1
     }^{2}x_{3}-x_{2}^{2},\,x_{2
     }^{4}-x_{1}x_{2}^{2}-x_{3}x
     _{4}^{2},\,x_{1}^{2}x_{2}^{
     2}-x_{1}^{3}-x_{4}^{2}\right)
\subset \CC^4
\label{eq:X2Ex}
\end{equation}

\begin{equation}
X_3=\bV\left((x_1^3-2x_1^2x_4+x_1^3+x_5^2x_4+x_2^2x_1-x_2x_1x_3+x_5^3\right)\subset \CC^5
\label{eq:X3Ex}
\end{equation}

{\small
\begin{equation}
 \begin{split}
  X_4=\bV\Big(&x_1^2x_2^2x_3 - 2x_1^2x_2x_3x_4 - 2x_1^2x_2x_3x_5 + x_1^2x_3x_4^2 + 2x_1^2x_3x_4x_5 + x_1^2x_3x_5^2- 2x_1x_2^3x_3 + \\ &4x_1x_2^2x_3x_4 + 4x_1x_2^2x_3x_5 - x_1x_2^2x_4 - 2x_1x_2x_3x_4^2 - 6x_1x_2x_3x_4x_5 - 2x_1x_2x_3x_5^2 + \\ &x_1x_2x_4^2 + x_1x_2x_4x_5 + 2x_1x_3x_4^2x_5 + 2x_1x_3x_4x_5^2 - x_1x_4^2x_5 + x_2^4x_3 - 2x_2^3x_3x_4 - \\ &2x_2^3x_3x_5 + x_2^3x_4 + x_2^2x_3x_4^2 + 4x_2^2x_3x_4x_5 + x_2^2x_3x_5^2 - 2x_2^2x_4^2 - x_2^2x_4x_5 - 2x_2x_3x_4^2x_5 - \\ & 2x_2x_3x_4x_5^2 + x_2x_4^3 + 2x_2x_4^2x_5 + x_3x_4^2x_5^2 - x_4^3x_5\Big)\subset\CC^5
  \end{split}
\label{eq:X4Ex}
\end{equation}}

Examples $ X_5$ and $X_6$ are also constructed to have non-trivial Whitney stratifications.

\begin{equation}
X_5=\bV\left(x_{6}x_{7}^{2}+x_{7}^{3}+x_{6}^{2}-x_{1}x_{7}-x_{7}^{2}+x_{2},\,x_{1}^{2}x_{2}^{2}-x_{1}^{3}+x_{1}x_{7}^{2}-x_{3}x_{7}^{2}-x_{5
     }^{2}\right)
\subset \CC^7 \label{eq:X5Ex}
\end{equation}

{\small
\begin{equation}\begin{split}
X_6=\bV\Big(& 3\,x_{2}x_{3}x_{5}+2\,x_{1}x_{4}x_{5}-x_{2}x_{4}x_{5}-x_{3}x_{4}x_{5}-2\,x_{1}x_{5}x_{6}-5\,x_{3}x_{5}x_{6}+2\,x_{4}x_{5}x_{6
      },\\&x_{2}x_{4}^{2}-x_{1}x_{4}x_{5}-3\,x_{2}x_{3}x_{6}+x_{1}x_{5}x_{6}+3\,x_{3}x_{5}x_{6}-x_{4}x_{5}x_{6},\\&\,x_{2}^{2}x_{4}-x_{1}x_{2}x_{5}+2\,x_{1}x_{2
      }x_{6}-x_{2}x_{3}x_{6}-2\,x_{2}x_{4}x_{6}+x_{3}x_{5}x_{6},\\&\,3\,x_{2}^{2}x_{3}+2\,x_{1}x_{2}x_{4}-x_{2}x_{3}x_{4}-x_{1}x_{2}x_{5}-6\,x_{2}x_{3}x_{6}+x
      _{3}x_{5}x_{6}\Big)
\subset \CC^6
 \end{split}
\label{eq:X6Ex}
\end{equation}}

These examples exhibit varying levels of complexity in terms of degree of generators, number of variables, density of polynomials, etc. Our next example $X_7$ is the determinantal variety defined by the $3\times 3$ minors of the matrix \[\left[\!\begin{array}{cccc}
      \vphantom{\left\{-1\right\}}x_{1}&x_{4}&x_{7}&x_{6}\\
      \vphantom{\left\{-1\right\}}x_{2}&x_{5}&x_{5}&x_{7}\\
      \vphantom{\left\{-1\right\}}x_{3}&x_{3}&x_{4}&x_{1}
      \end{array}\!\right].\] 
 Explicitly, we have
 {\small
 \begin{equation}
 \begin{split}
X_7=\bV(&x_{1}x_{4}x_{5}-x_{3}x_{5}x_{6}+x_{4}x_{5}x_{6}-x_{4
     }^{2}x_{7}-x_{1}x_{5}x_{7}+x_{3}x_{7}^{2}, \\ &x_{1}^{2}x_{5}+x_{2}x_{4}x_{6}-x
     _{3}x_{5}x_{6}-x_{1}x_{2}x_{7}-x_{1}x_{4}x_{7}+x_{3}x_{7}^{2}, \\ &x_{2}x_{4}^{2
     }+x_{1}x_{3}x_{5}-x_{3}x_{4}x_{5}-x_{3}x_{5}x_{6}+x_{4}x_{5}x_{6}-x_{2}x_{3}x
     _{7}-x_{4}^{2}x_{7}-x_{1}x_{5}x_{7}+x_{3}x_{5}x_{7}+x_{3}x_{7}^{2},\\ &x_{1}x_{
     2}x_{4}-x_{2}x_{3}x_{6}+x_{2}x_{4}x_{6}-x_{1}x_{2}x_{7}+x_{1}x_{3}x_{7}-x_{1}
     x_{4}x_{7}-x_{3}x_{4}x_{7}+x_{3}x_{7}^{2} )
\subset \CC^7
 \end{split}
\label{eq:X7Ex}
\end{equation}}
    
      The last two varieties in Table \ref{tab:example} have Whitney stratifications given by iterated singular loci; these have been included as control examples. The first of these is smooth and has been obtained by intersecting 2 random quadrics in 10 variables. The second one, called the {\em Barth Sextic} is naturally written in 3 variables, however we require an additional variable to represent the golden ratio $\frac{\sqrt{5}+1}{2}$ algebraically. Thus, our version of the sextic is represented by an ideal of degree 12 in 4 variables. 

\begin{table}[htb]
  \centering
  \begin{tabular}{l|cc|cc|cc|||cc}        
  \toprule
 {\bf Input} & \multicolumn{2}{c|}{\bf Assoc.~Primes~\cite{hn}} & \multicolumn{2}{c|}{\bf Block Order \cite{helmer2024new}} & \multicolumn{2}{c|||}{\bf WhitneyPolar}& {\bf CAD} \\ 
 \midrule
 ~& $\mathbb{F}_{32749}$ & $\QQ$ & $\mathbb{F}_{32749}$ & $\QQ$ & $\mathbb{F}_{32749}$ & $\QQ$ \\
 \midrule 
Whitney Umbrella & 0.07s &0.07s &\bf 0.05s & \bf 0.05s& 0.06s &0.06s &0.4s \\
Whitney Cusp & 0.09s & 0.09s & 0.07s & 0.07s& \bf 0.06s &\bf 0.06s  &0.4s\\
 $X_1$, see \eqref{eq:X1Ex}& 4.8s & 2.9s  & 4.5s & 2.7s & \bf 0.15s & 2.5s  & 0.4s \\
  $X_2$, see \eqref{eq:X2Ex}&  9.5s & 58.9s  & 9.2s & 58.3s  & \bf 0.1s & 0.2s &0.4s \\
 $X_3$, see \eqref{eq:X3Ex}&  1.1s & 1.6s  & 0.6s & 0.9s  & \bf 0.2s & 0.3s  &0.9s \\
 $X_4$, see \eqref{eq:X4Ex}&   $m$ &  $m$ & 20.7s & 747.2s &\bf 1.9s & 3.4s &   $m$\\
         $X_5$, see \eqref{eq:X5Ex}&  $m$ &  $m$ &  $m$&  $m$  & \bf 1.6s & 77.2s  &  $t$ \\
      $X_6$, see \eqref{eq:X6Ex}&  $t$  &  $t$   &  $t$&  $t$ & \bf 1.4s & 3.9s  &  $m$ \\
           $X_7$, see \eqref{eq:X7Ex}&  $t$ & $t$  & $t$ & $t$  & \bf 52.5s & 208.8s  & $m$\\
  Smooth & \bf 0.08s & 0.4s  & \bf 0.08s & 0.4s & \bf 0.08s & 0.4s &  $t$ \\
    Barth Sextic& \bf 0.6s &\bf 0.6s   &\bf 0.6s &\bf 0.6s & \bf0.6s &\bf 0.6s & 12717.7s \\
    \bottomrule
  \end{tabular}

    \caption{The first three columns show run times of various practical Whitney stratification algorithms on several example varieties; the run time for Cylindrical Algebraic Decomposition (CAD) computation is also included in the last column. Computations which failed to terminate after 11 hours are denoted $t$. Computations which  exhausted 64GB of memory and terminated are denoted $m$. For each run we highlight the fastest algorithm in {\bf bold}. 
    } 
  \label{tab:example}
\end{table}

\subsection{The Results}

\normalsize
Regarding the experiments in Table \ref{tab:example}, we note that the practical benefits of the new {\bf WhitneyPolar} algorithm become particularly apparent on larger examples with denser polynomials, e.g.~$X_4$, $X_5$, $X_6$ and $X_7$. This is true both in comparison to the other Whitney stratification methods from \cite{helmer2024new,hn}, as well as when compared to alternative methods to study systems of real polynomials such as Mathematica's highly refined CAD implementation.

\bibliographystyle{abbrv}
	\bibliography{library}
\end{document}